\documentclass[9pt,a4paper,english]{extarticle}

\pdfoutput=1
\usepackage{abstract}
\usepackage{amssymb}
\usepackage{amsmath}
\usepackage{amsthm}
\usepackage[boxed, vlined, linesnumbered, resetcount]{algorithm2e} 
\usepackage{bm}
\usepackage{datenumber}
\usepackage{datetime}
\usepackage{float}
\usepackage{listings}
\usepackage{mathtools}
\usepackage{multirow}
\usepackage[numbers, sort&compress]{natbib}
\usepackage[super]{nth}
\usepackage{physics}
\usepackage[scr,scaled=1.1]{rsfso}
\usepackage{subfigure}
\usepackage[binary-units=true]{siunitx}
\usepackage{todonotes}
\usepackage{hyperref}
\usepackage{caption}
\usepackage[top=20mm,bottom=20mm,left=20mm, right=20mm]{geometry}

\usepackage{xcolor}

\newfloat{lstfloat}{htbp}{lop} 

\SetAlCapSty{} 
\SetAlCapSkip{1em} 
\SetAlCapNameFnt{}
\SetAlCapFnt{}
\SetAlgoCaptionSeparator{:}

\lstdefinestyle{C}{
    language=C,
    basicstyle=\small\ttfamily,
    keywordstyle=\small\ttfamily,
    morekeywords={omp,simd,reduction,simdlen,declare,inline,bool,restrict,half},
    otherkeywords={\#pragma,\_\_fp16},
    frame = single,
    captionpos=b,
    abovecaptionskip=1em,
}

\DeclareMathOperator*{\argmin}{argmin} 

\newtheorem{theorem}{Theorem}[section]
\newtheorem{corollary}[theorem]{Corollary}
\newtheorem{lemma}[theorem]{Lemma}

\title{\Huge Approximating inverse cumulative distribution functions to produce approximate random variables}

\author{
\href{mailto:mike.giles@maths.ox.ac.uk}{Michael Giles}%
\thanks{\href{mailto:mike.giles@maths.ox.ac.uk}%
{\texttt{mike.giles@maths.ox.ac.uk}}} 
\and 
\href{mailto:oliver.sheridan-methven@hotmail.co.uk}{Oliver Sheridan-Methven}%
\thanks{\href{mailto:oliver.sheridan-methven@hotmail.co.uk}%
{\texttt{oliver.sheridan-methven@hotmail.co.uk}}}
}

\date{
Mathematical Institute, Oxford University, UK\\[1em]
\datedayname\ \nth{\number\day} \monthname\  \number\year}
\begin{document}

\maketitle

\begin{abstract}
For random variables produced through the inverse transform method, approximate random variables are introduced, which are produced by approximations to a distribution's inverse cumulative distribution function. These approximations are designed to be computationally inexpensive, and much cheaper than library functions which are exact to within machine precision, and thus highly suitable for use in Monte Carlo simulations. The approximation errors they introduce can then be eliminated through use of the multilevel Monte Carlo method.  Two approximations are presented for the Gaussian distribution: a piecewise constant on equally spaced intervals, and a piecewise linear using geometrically decaying intervals. The errors of the approximations are bounded and the convergence demonstrated, and the computational savings measured for C and C++ implementations. Implementations tailored for Intel and Arm hardwares are inspected, alongside hardware agnostic implementations built using OpenMP. The savings are incorporated into a nested multilevel Monte Carlo framework with the Euler-Maruyama scheme to exploit the speed ups without losing accuracy, offering speed ups by a factor of 5--7. These ideas are empirically extended to the Milstein scheme, and the non central $ \chi^2 $ distribution for the Cox-Ingersoll-Ross process, offering speed-ups of a factor of 250 or more.  

\begin{description}
\item[Keywords:] approximations, random variables, inverse cumulative distribution functions, random number generation, the Gaussian distribution, geometric Brownian motion, the Cox-Ingersoll-Ross process, the non central $ \chi^2 $ distribution, multilevel Monte Carlo, the Euler-Maruyama scheme, the Milstein scheme, and high performance computing.
\item[MSC subject classification:] 	65C10, 41A10, 65D15, 65C05, 62E17, 65Y20, 60H35, and 65C30.
\end{description}
\end{abstract}

\section{Introduction}
\label{sec:introduction}

Random number generation is a fundamental building block of a wide range of computational tasks, in particular financial simulations \citep{glasserman2013monte,asmussen2007stochastic,joy1996quasi,xu2015high}. A frequent performance bottleneck is the generation of random numbers from a specific statistical distribution such as in Monte Carlo simulations.  While computers have many excellent and fast implementations of random number generators for the uniform distribution on the interval $(0,1)$, sampling random variables from a generic distribution is often computationally much more expensive.

We consider random variables produced using the inverse transform method \citep{glasserman2013monte}, which enables sampling from any uni-variate distribution, and thus is very widely applicable. Additionally, the inverse transform method is crucial for quasi-Monte Carlo simulations \citep{giles2009multilevel_qmc,lecuyer2016randomized} as no samples are rejected and the low-discrepancy property is preserved \citep{tezuka1995uniform}; such applications are common in financial simulations \citep{joy1996quasi,xu2015high}. Furthermore, we demonstrate how the inverse transform method is particularly well suited to analysis, and that it naturally provides a coupling mechanism for multilevel Monte Carlo applications for a range of stochastic processes, statistical distributions, and numerical schemes. 

Our analysis focuses on the Gaussian distribution (a.k.a.\ the Normal distribution), and the motivations are threefold. Firstly, the distribution is representative of several continuous distributions, and due to the central limit theorem is often the limiting case. Secondly, it is analytically tractable and often admits exact results or is amenable to approximation. Lastly, it is ubiquitous in both academic analysis and scientific computation, with its role cemented within It\^{o} calculus and financial simulations. 

To produce Gaussian random variables will require the Gaussian distribution's inverse cumulative distribution function. Constructing approximations accurate to machine precision has long been investigated by the scientific community \citep{hastings1955approximations,evans1974algorithm70,beasley1985percentage,wichura1988algorithm,marsaglia1994rapid,giles2011approximating}, where the \textit{de facto} routine implemented in most libraries is by \citet{wichura1988algorithm}. While some applications may require such accurate approximations, for many applications such accuracy is excessive and unnecessarily costly, as is typically the case in Monte Carlo simulations. 

To alleviate the cost of exactly sampling from the Gaussian distribution, a popular circumvention is to substitute these samples with random variables with similar statistics. The bulk of such schemes follow a moment matching procedure, where the most well known is to use Rademacher random variables (which take the values $ \pm 1 $ with equal probability \citep[page~XXXII]{kloeden1999numerical}, giving rise to the weak Euler-Maruyama scheme), matching the mean and variance. Another is to sum twelve uniform random variables and subtract the mean \citep[page~500]{munk2011fixed}, which also matches the mean and variance, and is still computationally cheap. The most recent work in this direction is by \citet{muller2015improving}, who produce either a three or four point distribution, where the probability mass is positioned so the resulting distribution's moments match the lowest few moments of the Gaussian distribution; as in this paper, they combine this with use of the multilevel Monte Carlo (MLMC) method but in a way which is not directly comparable.

The direction we follow in this paper is more closely aligned to the work by \citet{giles2019random_quadrature,giles2019random_multilevel}, whose analysis proposes a cost model for producing the individual random bits constituting a uniform random number. They truncate their uniforms to a fixed number of bits and then add a small offset before using the inverse transform method. The nett result from this is to produce a piecewise constant approximation, where the intervals are all of equal width, and the values are the midpoint values of the respective intervals. 

The work we present directly replaces random variables produced from the inverse transform method using the exact inverse cumulative distribution function, with those produced using an approximation to the inverse cumulative distribution function. While this is primarily motivated by computational savings, our framework encompasses the distributions produced from the various moment matching schemes and the truncated bit schemes. 

Having a framework capable of producing various such distributions has several benefits. The first is that by refining our approximation, we can construct distributions resembling the exact distribution to an arbitrary fidelity. This allows for a trade-off between computational savings, and a lower degree of variance between the exact distribution and its approximation. This naturally  introduces two tiers of simulations: those using a cheap but approximate distribution, and those using an expensive but near-exact distribution. This immediately facilitates the multilevel Monte Carlo setting of \citet{giles2008multilevel}, where fidelity and cost are balanced to minimise the computational time. As an example, we will see in section~\ref{sec:multilevel_monte_carlo} that Rademacher random variables, while very cheap, are too crude to exploit any savings possible with multilevel Monte Carlo, whereas our higher fidelity approximations can fully exploit the possible savings. 

The second benefit of our approach is that while the approximations are  specified mathematically, their implementations are left unspecified. This flexibility facilitates constructing approximations which can be tailored to a specific hardware or architecture. We will present two approximations, whose implementations can gain speed by transitioning the work load from primarily using the floating point processing units to instead exploiting the cache hierarchy. Further to this, our approximations are designed with vector hardware in mind, and are non branching, and thus suitable for implementation using single instruction multiple data (SIMD) instructions, (including Arm's new scalable vector extension (SVE) instruction set). Furthermore, for hardware with very large vectors, such as the \SI{512}{\bit} wide vectors on Intel's AVX-512 and Fujitsu's Arm-based A64FX (such as those in the new Fugaku supercomputer), we demonstrate implementations unrivalled in their computational performance on the latest hardware. Previous work using low precision bit wise approximations targetted at reconfigurable field programmable gate arrays has previously motivated the related works by \citet{brugger2014mixed}, \citet{omland2015exploiting}, and \citet{cheung2007hardware}.

We primarily focus our attention on the analysis and implementation of two approximations: a piecewise constant approximation using equally sized intervals, and a piecewise linear approximation on a geometric sequence of intervals. The former is an extension of the work by \citet[theorem~1]{giles2019random_quadrature} to higher moments, while the latter is a novel analysis, capable of both the highest speeds and fidelities. Although we will demonstrate how these are incorporated into a multilevel Monte Carlo framework, the subsequent analysis is performed by \citeauthor{giles2020approximate} \citep{giles2020approximate,sheridan2020nested} and omitted from this work. 

Having outlined our approximation framework and the incorporation of approximate random variables within a nested multilevel Monte Carlo scheme, we quantify the savings a practitioner can expect. For a geometric Brownian motion process from the Black-Scholes model \citep{black1973pricing} using Gaussian random variables, savings of a factor of 5--7 are possible. Furthermore, for a Cox-Ingersoll-Ross process \citep{cox1985theory} using non-central $ \chi^2 $ random variables, savings of a factor of 250 or more are possible. These savings are benchmarked against the highly optimised Intel C/C++ MKL library for the Gaussian distribution, and against the CDFLIB library \citep{brown1994dcdflib,burkardt2020cdflib} in C and Boost library \citep{boost2020library} in C++ for the non-central $ \chi^2 $ distribution. 

Section~\ref{sec:approximate_gaussian_random_variables} introduces and analyses our two approximations, providing bounds on the error of the approximations. Section~\ref{sec:high_performance_impementations} discusses the production of high performance implementations, the code for which is collected into a central repository maintained by \citet{sheridan2020approximate_random,sheridan2020approximate_inverse}. Section~\ref{sec:multilevel_monte_carlo} introduces multilevel Monte Carlo as a natural application for our approximations, and demonstrates the possible computational savings. Section~\ref{sec:the_non_central_chi_squared_distribution} extends our approximations to the non central $ \chi^2 $ distribution, representing a very expensive parametrised distribution which arises in simulations of the Cox-Ingersoll-Ross process. Lastly, section~\ref{sec:conclusions} presents the conclusions from this work. 

\section{Approximate Gaussian random variables}
\label{sec:approximate_gaussian_random_variables}

The inverse transform method \citep[2.2.1]{glasserman2013monte} produce univariate random variable samples from a desired distribution by first generating a uniform random variable $U$ on the unit interval $(0,1)$, and then computing $C^{-1}(U)$, where $C^{-1}$ is the inverse of the distribution's cumulative distribution function (sometimes called the percentile or quantile function). We will focus on sampling from the Gaussian distribution, whose inverse cumulative distribution function we denote by $ \Phi^{-1} \colon (0, 1) \to \mathbb{R} $ (some authors use $ N^{-1} $), and similarly whose cumulative distribution function and probability density function we denoted by $ \Phi $ and $ \phi $ respectively. 

Our key proposal is to use the inverse transform method with an approximation to the inverse cumulative distribution function. As the resulting distribution will not exactly match the desired distribution, we call random variables produced in this way \emph{approximate random variables}, and those without the approximation as \emph{exact random variables} for added clarity. In general, we will denote exact Gaussian random variables by $ Z $ and approximate Gaussian random variables by $ \widetilde{Z} $. The key motivation for introducing approximate random variables is that they are computationally cheaper to generate than exact random variables. Consequently, our key motivating criterion in forming approximations will be a simple mathematical construction, hoping this fosters fast implementations. As such there is a trade-off between simplicity and fidelity, where we will primarily be targetting simplicity. 

In this section, we present two approximation types: a piecewise constant, and a piecewise linear. For both we will bound the $ L^p $ error, focusing on their mathematical constructions and analyses. Their implementations and utilisation will be detailed in sections~\ref{sec:high_performance_impementations} and \ref{sec:multilevel_monte_carlo}. Both analyses will share and frequently use approximations for moments and tail values of the Gaussian distribution, which we gather together in section~\ref{sec:approximating_tail_values_and_high_order_moments}. Thereafter, the piecewise constant and linear approximations are analysed in sections~\ref{sec:piecewise_constant_approximations_on_equal_intervals} and \ref{sec:piecewise_linear_approximations_on_geometric_intervals}.

\subsection{Approximating tail values and high order moments}
\label{sec:approximating_tail_values_and_high_order_moments}

We require bounds on the approximations of tail values and high order moments. We use the notation of \citet{giles2019random_quadrature} that $ f(z) \approx g(z) $ denotes $ \lim_{z\to\infty} \tfrac{f(z)}{g(z)} = 1 $. Our key results are lemmas~\ref{lemma:approximate_tail_values} and \ref{lemma:approximate_moments} which bound the tail values and high order moments. Lemma~\ref{lemma:approximate_tail_values} is an extension of a similar result by \citet[lemma~7]{giles2019random_quadrature}, extended to give enclosing bounds. Similarly, lemma~\ref{lemma:approximate_moments} is partly an extension of a result by \citet[lemma~9]{giles2019random_quadrature}, but extended to arbitrarily high moments rather than just the second. As such, neither of these lemmas are particularly noteworthy in themselves and their proofs resemble work by \citet[appendix~A]{giles2019random_quadrature}. Similarly, our resulting error bounds for the piecewise constant approximation in section~\ref{sec:piecewise_constant_approximations_on_equal_intervals} will closely resemble a related result by \citet[theorem~1]{giles2019random_quadrature}. However, our main result for the piecewise linear approximation in section~\ref{sec:piecewise_linear_approximations_on_geometric_intervals} is novel and will require these results, and thus we include them here primarily for completeness. 

\begin{lemma}
\label{lemma:approximate_tail_values}
Defining $ z_q \coloneqq \Phi^{-1}(1 {-} 2^{-q})$, then for $ q \geq  2 $ we have the bounds $ 1 < \tfrac{2^q\phi(z_q)}{z_q} < (1 - z_q^{-2})^{-1}$  and for  $ q \geq 10 $
$ \sqrt{q \log 4 - \log(q \pi (81/32) \log 4)} < z_q  < \sqrt{q \log 4}$.
\end{lemma}

\begin{proof}
  $ 2^{-q} = 1 - \Phi(z_q) = \int_{z_q}^{\infty} \phi(s) \dd{s} $.
  Integrating by parts using $ \phi'(z) =  -z \phi(z) $ gives
\[
\int_{z_q}^{\infty} \phi(s) \dd{s}
\ =\ \phi(z_q) \tfrac{1}{z_q} - \int_{z_q}^{\infty} \tfrac{1}{s^2} \phi(s) \dd{s}
\ =\ \phi(z_q) (\tfrac{1}{z_q} - \tfrac{1}{z^3_q}) + \int_{z_q}^{\infty} \tfrac{1}{4s^4} \phi(s) \dd{s},
\]
which proves that $\phi(z_q) \tfrac{1}{z_q} > 2^{-q} > \phi(z_q)(\tfrac{1}{z_q} - \tfrac{1}{z_q^3}) $ and hence we obtain the first inequality.

For $q \geq 10$, $z_q > 3$ and so $(1{-}z_q^{-2})^{-1} < \tfrac{9}{8}$. The inequality $\tfrac{9}{8} 2^{-q} z_q > \phi(z_q) > z_q$ then gives
\[
\sqrt{q \log 4 - \log (2\pi \tfrac{81}{64} z_q^2)} < z_q <
\sqrt{q \log 4 - \log (2\pi z_q^2)}.
\]
Inserting
the upper bound $z_q<\sqrt{q \log 4}$ into the left hand term, and
the lower bound $z_q^2 > 1/2\pi$ in the right hand term,
gives the desired second inequality.
\end{proof}

\begin{lemma}
\label{lemma:approximate_moments}
For integers $ p \geq 2 $ we have $ \int_{0}^{z} \phi(s)^{1-p} \dd{s} \approx \tfrac{\phi(z)^{1-p}}{(p-1)z}  $ and $ \int_{z}^{\infty} (s-z)^p\phi(s) \dd{s} \approx \tfrac{p!\phi(z)}{z^{p+1}} $ as $z\rightarrow\infty$.
\end{lemma}

\begin{proof}
Applying L'H\^{o}pital's rule gives
\begin{equation*}
\lim_{z \to \infty} \dfrac{\int_{0}^{z} \phi(s)^{1-p} \dd{s}}{\left(\dfrac{\phi(z)^{1-p}}{(p-1)z}\right)} 
=  \lim_{z \to \infty} \dfrac{\phi(z)^{1-p}}{\phi(z)^{1-p}\left(1 - \dfrac{1}{(p-1)z^2} \right)} 
= 1.
\end{equation*}
Similarly, applying L'H\^{o}pital's rule $p+1$ times gives
\[
\lim_{z \to \infty} \dfrac{\int_z^\infty (s{-}z)^p \phi(s) \dd{s}}{\left(\dfrac{p!\, \phi(z)}{z^{p+1}}\right)} 
=  \lim_{z \to \infty} \dfrac{(-1)^p p!\, \phi(z)}{(-1)^p p!\, \phi(z) \,(1 + O(1/z^2))}
= 1.\qedhere
\]
\end{proof}

\subsection{Piecewise constant approximations on equal intervals}
\label{sec:piecewise_constant_approximations_on_equal_intervals}

Mathematically it is straightforward to motivate a piecewise constant approximation as the simplest possible approximation to use, especially using equally spaced intervals. As the range of values will go from a continuous to a discrete set, we say the distribution has become \emph{quantised} and denote our approximation as $ Q \approx \Phi^{-1} $, where for a uniform random variable $ U \sim \mathcal{U}(0, 1)$ we have $ Z \coloneqq \Phi^{-1}(U) $ and $ \widetilde{Z} \coloneqq Q(U) $. A preview of such an approximation is shown in figure~\ref{fig:piecewise_constant_gaussian_approximation}, where the error will be measured using the $ L^p $ norm $ \lVert f \rVert_p \coloneqq (\int_0^1 \abs{f(u)}^p \dd{u})^{1/p} $.

\begin{figure}[htb]
\centering

\hfil
\subfigure[A piecewise constant approximation using 8 intervals.\label{fig:piecewise_constant_gaussian_approximation}]{\includegraphics{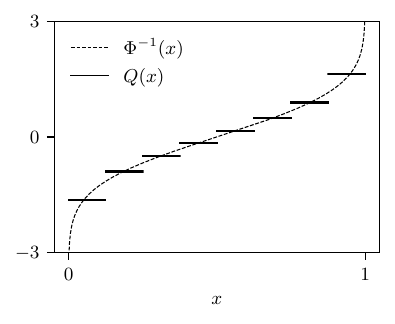}}\hfil
\subfigure[The $ L^p $ error and the bound from theorem~\ref{thm:piecewise_constant_approximation_error}.\label{fig:piecewise_constant_gaussian_approximation_error}]{\includegraphics{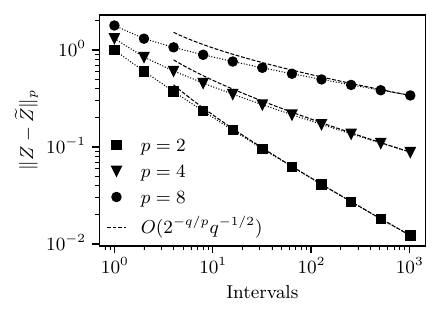}}\hfil

\caption{The piecewise constant approximation $ \widetilde{Z}_k^{L^1} $ from corollary~\ref{cor:piecewise_constant_constructions} with equally spaced intervals, and its error.}
\label{fig:piecewise_constant_approximation}
\end{figure}

Before presenting theorem~\ref{thm:piecewise_constant_approximation_error}, we can briefly comment on the error seen in figure~\ref{fig:piecewise_constant_gaussian_approximation_error}. Specifically looking at the root mean squared error (RMSE), corresponding to the $ L^2 $ norm, we can see that increasing the number of intervals from 2 to $ 10^3 $ gives a drop of $ 10^2 $ in the RMSE. For our multilevel Monte Carlo applications in section~\ref{sec:multilevel_monte_carlo}, having approximately 1000 intervals gives a very good fidelity, whereas Rademacher random variables have a very low fidelity. Being able to achieve a reasonable fidelity from our approximation ensures that we achieve the largest portion of the possible computational savings offered by our approximations. 

As we have already mentioned, the piecewise constant approximation is closely related to the resulting approximation produced by \citet{giles2019random_quadrature}, whose approximation arises from considering uniform random variables truncated to a finite number of bits of precision. Thus our main result from this section, theorem~\ref{thm:piecewise_constant_approximation_error}, closely resembles a related result by \citet[theorem~1]{giles2019random_quadrature}. To put our extension into context, we paraphrase the similar result from \citet{giles2019random_quadrature}, which is that for a piecewise constant approximation using $ 2^q $ intervals, for some integer $ q \gg 1 $, then for constant values equal to each interval's midpoint value they have $ \lVert Z - \widetilde{Z}\rVert_2^2 = O(2^{-q}q^{-1}) $. Our result from theorem~\ref{thm:piecewise_constant_approximation_error} extends this to $ \lVert Z - \widetilde{Z}\rVert_p^p = O(2^{-q} q^{-p/2}) $ for $ p \geq 2 $, and numerous other possible constant values other than the midpoint's. Our result enables us to increase the order of the error to arbitrarily high norms and is interesting in its own right. It shows that as the intervals become increasing small (corresponding to $ q \to \infty $), the dominant term effecting the $ L^p $ error is the geometric decay $ O(2^{-q/p}) $ due to the error in the two end intervals, and thus the convergence exists but is slower in higher norms, with the polynomial $ O(q^{-1/2}) $ term being comparatively negligible (as we can see in figure~\ref{fig:piecewise_constant_gaussian_approximation_error}). Additionally, in the related analysis incorporating approximate random variables into a nested multilevel Monte Carlo framework by \citeauthor{giles2020approximate} \citep{giles2020approximate,sheridan2020nested}, their bounds on the variance of the multilevel Monte Carlo correction term (discussed more in section~\ref{sec:multilevel_monte_carlo}) rely on the existence of the $ L^p $ error for $ p > 2 $. Hence, while this strengthening of the result may appear only slight, it is crucial for nested multilevel Monte Carlo. 

We can now present our main result concerning piecewise constant approximations, namely theorem~\ref{thm:piecewise_constant_approximation_error}. In this we will leave the interval values largely unspecified, and later demonstrate in corollary~\ref{cor:piecewise_constant_constructions} several choices fit within the scope of theorem~\ref{thm:piecewise_constant_approximation_error}.

\begin{theorem}
\label{thm:piecewise_constant_approximation_error}
Let a piecewise constant approximation $ Q \approx \Phi^{-1} $ use $ 2^q $ equally spaced intervals for some integer $ q > 1 $. Denote the intervals $ I_k \coloneqq (k2^{-q}, (k+1)2^{-q}) \equiv (u_k, u_{k+1}) $ for $ k \in \{0, 1, 2, \ldots, K\} $ where $ K \equiv 2^q - 1 $. On each interval the approximation constant is $ Q_k \coloneqq Q(u) $ for $ u \in I_k $. We assume there exists a constant $ C $ independent of $ q $ such that:
\begin{enumerate}
\item \label{con:symmetry} $ Q_k = -Q_{K-k} $ for $ k \in \{0,1,2,\ldots,K\} $. 
\item \label{con:intermediate_values} $ \Phi^{-1}(u_k) \leq Q_k \leq \Phi^{-1}(u_{k+1}) $ for $ k \in \{1,2, \ldots, K-1\} $.
\item \label{con:bounded_asymptotic_growth} $ \Phi^{-1}(u_K) \leq Q_K \leq \Phi^{-1}(u_K) + Cq^{-1/2} $.
\end{enumerate}
Then for any even integer $ p \geq 2 $ we have $ \lVert  Q - \Phi^{-1} \rVert_p^p = O(2^{-q}q^{-p/2})$ as $ q \rightarrow \infty $.
\end{theorem}

\begin{proof} Defining $ A_k \coloneqq \int_{u_k}^{u_{k+1}} \lvert\Phi^{-1}(u) - Q_k\rvert^p \dd{u} $ for $ k \in \{0,1, \ldots, K\} $, then we have $ \lVert Q - \Phi^{-1} \rVert_p^p = \sum_{k=0}^{K} A_k $.  We use condition (\ref{con:symmetry}) to reduce our considerations to the domain $ (\tfrac{1}{2}, 1) $, where $ \Phi^{-1} $ is strictly positive and convex, and thus obtain $ \lVert Q - \Phi^{-1} \rVert_p^p = 2 \sum_{k=2^{q-1}}^{K-1} A_k + 2 A_K $. As $ \Phi^{-1} $ is convex, then from the intermediate value theorem there exists a $ \xi_k \in I_k $ such that $ Q_k = \Phi^{-1}(\xi_k) $. Noting that $ \dv{z} \Phi^{-1}(z) = \tfrac{1}{\phi(\Phi^{-1}(z))} $, then from the mean value theorem  for any $ u \in [u_k, \xi_k)$ there exists an $ \eta_k \in [u_k, \xi_k)$ such that $ \Phi^{-1}(u) - \Phi^{-1}(\xi_k) = \tfrac{u - \xi_k}{\phi(\Phi^{-1}(\eta_k))} $. Furthermore, as $ \phi $ is monotonically decreasing in $ (\tfrac{1}{2}, 1) $, then introducing $ z_k \coloneqq \Phi^{-1}(u_k) $ we have $ \lvert \Phi^{-1}(u) - Q_k\rvert \leq \tfrac{2^{-q}}{\phi(z_{k+1})}$. An identical argument follows for any $ u \in [\xi_k, u_{k+1})$ giving the same bound.  Using this to bound $ \sum_{k=2^{q-1}}^{K-1} A_k $ in our expression for $ \lVert Q - \Phi^{-1}\rVert_p^p $ gives
\begin{equation*}
\sum_{k=2^{q-1}}^{K-1} A_k 
\leq \sum_{k=2^{q-1}}^{K - 1} \left(\frac{2^{-q}}{\phi(z_{k+1})}\right)^p 2^{-q} 
\leq 2^{-qp - q} \sum_{k=2^{q-1}}^{K - 1}  \phi(z_{k+1})^{-p} 
\leq 2^{-pq} \int_{\frac{1}{2}}^{1 - 2^{-q}} \phi(\Phi^{-1}(u))^{-p} \dd{u} + 2^{-q(p+1)}\phi(z_K)^{-p},
\end{equation*}
where the last bound comes from considering a translated integral. Changing integration variables the integral becomes $  \int_{0}^{z_K} \phi(z)^{1-p} \dd{z} $, from which we can use lemma~\ref{lemma:approximate_moments} to give 
\begin{equation*}
\sum_{k=2^{q-1}}^{K-1} A_k 
\leq \frac{2^{-q} z_K^{-p}}{p-1} \left(\frac{2^q\phi(z_K)}{z_K}\right)^{1-p} + 2^{-q}z_K^{-p} \left(\frac{2^q\phi(z_K)}{z_K}\right)^{-p}
\leq 2^{-q} \left(\frac{p}{p-1}\right) z_K^{-p},
\end{equation*}
where the last bound follows from lemma~\ref{lemma:approximate_tail_values}. Turning our attention to the final interval's contribution $ A_K $, then using Jensen's inequality, condition~(\ref{con:bounded_asymptotic_growth}), and lemmas~\ref{lemma:approximate_tail_values} and \ref{lemma:approximate_moments} we obtain
\begin{equation*}
A_K 
\leq 2^{p-1} \int_{z_K}^{\infty} \lvert Q_K - z\rvert^p \phi(z) \dd{z} +  2^{p-1} \int_{u_K}^{1} \lvert Q_K - z_K \rvert^p \dd{u} 
\leq 2^{p-q-1} p! z_K^{-p} + 2^{p-q-1} C^p q^{-p/2}.
\end{equation*}
Combining our two bounds for $ \sum_{k=2^{q-1}}^{K-1} A_k $ and $ A_K $ into our expression for  $ \lVert Q - \Phi^{-1}\rVert_p^p $ we obtain
\begin{equation*}
\lVert Q - \Phi^{-1}\rVert_p^p 
\leq 2^{-q+1} \left(\frac{p}{p-1}\right) z_K^{-p} + 2^{p-q} p! z_K^{-p} + 2^{p-q} C^p q^{-p/2} 
\leq O(2^{-q}q^{-p/2}),
\end{equation*}
where the coefficients inside the $ O $-notation are only a function of $ p $ and not of $ q $. \qedhere
\end{proof}

\begin{corollary}
\label{cor:piecewise_constant_constructions}
For approximate Gaussian random variables $ \widetilde{Z} $ using constants $ \widetilde{Z}_k $ on intervals $ I_k $ constructed as either
\begin{equation*}
\label{eqt:approximate_normal_expected_value_construction}
\widetilde{Z}_k^\mathrm{E} \coloneqq \mathbb{E}(Z\mid \Phi(Z) \in I_k), 
\qquad 
\widetilde{Z}_k^\mathrm{C} \coloneqq {\Phi^{-1}}\left(\dfrac{\max I_k + \min I_k}{2}\right), 
\qquad \text{or} \qquad 
\widetilde{Z}_k^\mathrm{I} \coloneqq 
\begin{cases}
\Phi^{-1}(\min I_k) & \text{if } \min I_k \geq 0.5 \\
\Phi^{-1}(\max I_k) & \text{if } \max I_k < 0.5, 
\end{cases}
\end{equation*}
then all finite moments of $ \lvert\widetilde{Z}\rvert $ are uniformly bounded, and conditions~(\ref{con:symmetry}--\ref{con:bounded_asymptotic_growth}) of theorem~\ref{thm:piecewise_constant_approximation_error} are satisfied. 
\end{corollary}

\begin{proof}
As all three approximations are anti-symmetric about $u = 1/2$, and therefore all satisfy conditions~(\ref{con:symmetry}, we consider the domain $ (\tfrac{1}{2}, 1) $ where $ Z $ and $ \widetilde{Z} $ are both positive. Letting the interval $ I_k $ lie in this domain, then from Jensen's inequality and the law of iterated expectations we obtain 
\begin{equation*}
\mathbb{E}((\widetilde{Z}^{\mathrm{E}})^n) = \mathbb{E}((\mathbb{E}_{I_k}(Z \mid \Phi(Z) \in I_k))^n) \leq \mathbb{E}(\mathbb{E}_{I_k}(Z^n \mid \Phi(Z) \in I_k)) =  \mathbb{E}(Z^n) < \infty,
\end{equation*}
for any $ 1\leq n < \infty $, where the last inequality is a standard result. Mirroring this result to the full domain $ (0, 1) $, we can directly see that the moments of $ \lvert \widetilde{Z}^{\mathrm{E}} \rvert $ are uniformly bounded. Furthermore, as $ \Phi^{-1} $ is increasing and convex in $ [\tfrac{1}{2}, 1) $, then the value $ \widetilde{Z}_k^\mathrm{E} $ is an upper bound on $ \widetilde{Z}_k^\mathrm{C} $ and $ \widetilde{Z}_k^\mathrm{I} $, and so these too have uniformly bounded moments.

It is immediately clear that all three choices satisfy condition~\ref{con:intermediate_values}) and the lower bound in \ref{con:bounded_asymptotic_growth}). As $ \widetilde{Z}_k^\mathrm{E} $ is an upper bound for the other two choices, it will suffice to show that this satisfies the upper bound in \ref{con:bounded_asymptotic_growth}). Inspecting the difference between the smallest value in the final $ I_K $ interval, namely $ z_K $, and the average value $ Q_K = \mathbb{E}(Z \mid \Phi(z) \in I_K) $, we obtain
\begin{equation*}
Q_K - z_K = 2^q \int_{u_K}^{1} \left(\Phi^{-1}(u) - \Phi^{-1}(u_K)\right) \dd{u} = 2^q \int_{z_K}^{\infty} (z - z_K)\, \phi(z) \dd{z} <  2^q \dfrac{\phi(z_K)}{z_K^2} < \dfrac{1}{z_K} (1 - z_K^{-2})^{-1} \leq \dfrac{9}{8 z_K},
\end{equation*}
where the first inequality follows from lemma~\ref{lemma:approximate_moments}, the second from lemma~\ref{lemma:approximate_tail_values}, and the last uses $ (1 - z_K^{-2})^{-1} \leq \tfrac{9}{8} $ for $ q \geq 10 $. We bound the final $ \tfrac{1}{z_K} $ term using lemma~\ref{lemma:approximate_tail_values} for $q\geq  10$ to give $ \tfrac{1}{z_K} < 1.1 \, q^{-1/2}$ 
and therefore $ Q_K \leq z_K + 1.25 \,q^{-1/2} $. \qedhere
\end{proof}

We remark that the $ \widetilde{Z}_k^\mathrm{C} $ construction from corollary~\ref{cor:piecewise_constant_constructions} is the same as the value used in the analysis of \citet[(4)]{giles2019random_quadrature}. In any piecewise constant approximations we use, such as that in figure~\ref{fig:piecewise_constant_gaussian_approximation}, we will use the constants defined by  the $ \widetilde{Z}_k^\mathrm{E} $ construction from corollary~\ref{cor:piecewise_constant_constructions}. Furthermore, we can see that our bound from theorem~\ref{thm:piecewise_constant_approximation_error} appears tight in figure~\ref{fig:piecewise_constant_gaussian_approximation_error}.

\subsection{Piecewise linear approximations on geometric intervals}
\label{sec:piecewise_linear_approximations_on_geometric_intervals}

Looking at the piecewise constant approximation in figure~\ref{fig:piecewise_constant_gaussian_approximation}, it is clear there are two immediate improvements that can be made. The first is to use a piecewise linear approximation, which is considerably more appropriate for the central region. Secondly, the intervals should not be of equal sizes, but denser near the singularities at either end. We will make both these modifications in a single step, where we will construct a piecewise linear approximation with geometrically decreasing intervals which are dense near the singularities. For brevity we will denote this just as the piecewise linear approximation. An example piecewise linear approximation using 8 intervals is shown in figure~\ref{fig:piecewise_linear_gaussian_approximation}. The precise nature of the interval widths, and how the linear functions are fitted will be detailed shortly, but by direct comparison against figure~\ref{fig:piecewise_constant_gaussian_approximation} it is clear that the fidelity of a piecewise linear approximation is much better than the piecewise constant. 

The main result from this section will be theorem~\ref{thm:piecewise_linear_approximation_error}, which will bound the $ L^p $ error of our piecewise linear approximation. The proof will proceed in a similar fashion the proof of theorem~\ref{thm:piecewise_constant_approximation_error}, where we will bound the sum of the central intervals and the end intervals separately. For the central intervals we will use the Peano kernel theorem to bound the point wise error, and in the end intervals several results will be a mixture of exact results and bounds from lemmas~\ref{lemma:approximate_tail_values} and \ref{lemma:approximate_moments}.

\begin{figure}[htb]
\centering

\hfil
\subfigure[A piecewise linear approximation using 8  intervals.\label{fig:piecewise_linear_gaussian_approximation}]{\includegraphics{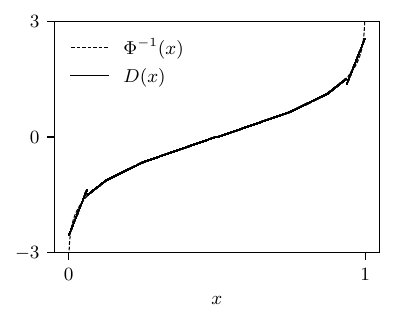}}\hfil
\subfigure[The $ L^2 $ error for various polynomial orders, with the number of intervals in $ (0, \tfrac{1}{2}) $ labeled.\label{fig:piecewise_linear_gaussian_approximation_error}]{\includegraphics{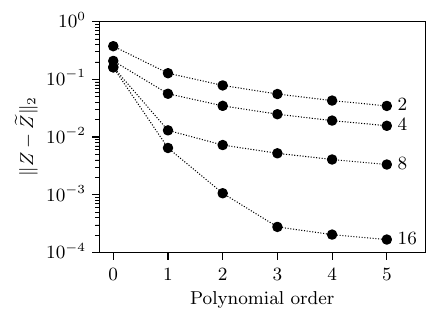}}\hfil

\caption{The piecewise linear approximation and its error from theorem~\ref{thm:piecewise_linear_approximation_error} with geometric intervals using $ r = \tfrac{1}{2} $.}
\label{fig:piecewise_linear_approximation}
\end{figure}

\begin{theorem}
\label{thm:piecewise_linear_approximation_error}
For an approximation $ D \approx \Phi^{-1}$ which is anti-symmetric about $ u=1/2$, with $ K $ intervals in $ (0, \tfrac{1}{2}) $, we define the $ k $-th interval $ I_k \coloneqq [\tfrac{r^k}{2}, \tfrac{r^{k-1}}{2})$ for $ k \in \{1,2,\ldots, K-1\} $ and $ I_K \coloneqq (0, \tfrac{r^{K-1}}{2}) $ for some decay rate $ r \in(0,1) $. Each interval uses a piecewise linear approximation $ D_k(u) \equiv D(u) $ for any $ u \in I_k $. The gradient and intercept in each interval is set by the $ L^2 $ minimisation $ D_k \coloneqq \argmin_{D' \in \mathcal{P}_1} \int_{I_k}\lvert\Phi^{-1}(u) - D'(u)\rvert^2 \dd{u} $ where $ \mathcal{P}_1 $  is the set of all 1-st order polynomials. Then we have for any $ 2 \leq p < \infty  $
\begin{equation*}
\lVert D - \Phi^{-1}\rVert_p^p 
= O((1-r)^{2p}) +
O(r^{K-1} {\log}^{-p/2}(r^{1-K}\sqrt{2/\pi}))  = O((1-r)^{2p}) +
o(r^{K-1}).
\end{equation*}
\end{theorem}

\begin{proof}
Considering the domain $ (0, \tfrac{1}{2}) $, we split the contribution into those from the intervals without the singularity, and that from the final interval with the singularity, where
\begin{equation*}
\lVert D - \Phi^{-1}\rVert_p^p   = 2  \sum_{k=1}^{K-1} \int_{I_k} \lvert D_k(u) - \Phi^{-1}(u)\rvert^p \dd{u}  + 2 \int_{I_K} \lvert D_K(u) - \Phi^{-1}(u)\rvert^p \dd{u},
\end{equation*}
where the factors of 2 correct for us only considering the lower half of the domain $ (0, 1) $. 

Beginning with the non-singular intervals, we express the pointwise error using the Peano kernel theorem \citep{iserles2009first,powell1981approximation}, which we will later bound. For notational simplicity, we denote the approximated function as $ f $, where $ f \equiv \Phi^{-1} $, and a given interval as $ [a,b] \equiv I_k  $. The $ L^2 $ optimal linear approximation is $ \alpha(f) + \beta(f) u $ for $ u \in [a,b] $ where $ \alpha $ and $ \beta $ are functionals. The point wise error $ f(u) - \alpha(f) - \beta(f)u $ is a linear mapping $ L $ acting on $ f $ where $ L(f)(u) \coloneqq  f(u) - \alpha(f) - \beta(f)u $. By construction $ L $ annihilates linear functions, so the Peano kernel is $ k(\xi; u) \coloneqq (u - \xi)^+ - \alpha((\cdot - \xi)^+) - \beta((\cdot - \xi)^+)u \equiv (u - \xi)^+ - \bar{\alpha}(\xi) - \bar{\beta}(\xi)u $ for $ \xi \in [a,b] $ where we defined $ \bar{\alpha}(\xi) \coloneqq \alpha((\cdot - \xi)^+) $ and similarly $ \bar{\beta}(\xi) $. The pointwise error is $ \varepsilon(u) \coloneqq L(f)(u) = \int_{a}^{b} k(\xi; u) f''(\xi) \dd{\xi} $.

To determine the intercept and gradient, we use that they are $ L^2 $ optimal, and so the functional derivatives of $ \int_a^b \lvert \varepsilon(u) \rvert^2 \dd{u} $ with respect to $ \alpha $ and $ \beta $ are zero, giving the simultaneous equations
\begin{equation*}
\alpha(f)(b-a) + \beta(f) \left(\dfrac{b^2 - a^2}{2}\right)  = \int_{a}^{b} f(u) \dd{u} 
\qquad \text{and} \qquad 
\alpha(f)\left(\dfrac{b^2 - a^2}{2}\right) + \beta(f) \left(\dfrac{b^3 - a^3}{3}\right)  = \int_{a}^{b} u f(u) \dd{u}.
\end{equation*}
It is important to notice that because we chose the $ L^2 $ norm, these are a set of linear simultaneous equations, and thus $ \alpha $ and $ \beta $ are linear functionals, thus showing that $ L $ is linear, (a requirement of the Peano kernel theorem). Evaluating these for the kernel function ($ f \to (\cdot - \xi)^+  $) gives
\begin{equation*}
\bar{\alpha}(\xi)  = - \dfrac{(b - \xi)^2 ((b+a)\xi - 2a^2)}{(b - a)^3} 
\qquad \text{and} \qquad 
\label{eqt:peano_kernel_coefficient}
\bar{\beta}(\xi) = \dfrac{(b - \xi)^2 (2\xi + b - 3a)}{(b - a)^3}.
\end{equation*}
Thus, the pointwise error is 
\begin{equation*}
\varepsilon(u)  = \int_{a}^{b} ((u - \xi)^+ - \bar{\alpha}(\xi) - \bar{\beta}(\xi) u ) f''(\xi) \dd{\xi} = (b - a)^2 \int_{0}^{1} \left((\widetilde{u} - \widetilde{\xi})^+ - (1 - \widetilde{\xi}^2)(\widetilde{\xi} + \widetilde{u})\right) f''((b - a)\widetilde{\xi} + a) \dd{\widetilde{\xi}},
\end{equation*}
where to achieve the last equality we rescaled our interval $ [a, b] \to [0, 1] $  and variables $ \eta \to \tilde{\eta}  $ where $ \tilde{\eta} \coloneqq \tfrac{\eta - a}{b - a} $. Taking the absolute value and applying Jensen's inequality immediately gives 
\begin{equation*}
\lvert \varepsilon(u) \rvert \leq - 6 (b - a)^2  {\dv[2]{\Phi^{-1}}{u}}(a) = - 6 \left(\dfrac{1-r}{r}\right)^2 a^2\, {\dv[2]{\Phi^{-1}}{u}}(a) \leq 2.59 \left(\dfrac{1-r}{r}\right)^2,
\end{equation*}
where for the first inequality we used that $ {\dv[2]{}{u}}\Phi^{-1} $ is maximal at the lower boundary, and for the second inequality we bound this by the maximum with respect to $ a $ (which is at $a \approx 0.177$). Using this expression for the pointwise error in our summation of the non-singular intervals gives (as $ r \to 1 $)
\begin{equation*}
 \sum_{k=1}^{K-1} \int_{I_k} \lvert D_k(u) - \Phi^{-1}(u)\rvert^p \dd{u}
= \sum_{k = 1}^{K - 1} \int_{I_k} \lvert\varepsilon(u)\rvert^p \dd{u}
\leq 2.59^p \left(\dfrac{1 - r}{r}\right)^{2p}  \int_{0}^{\frac{1}{2}} \dd{u}
= O((1-r)^{2p}).
\end{equation*}

We now consider the interval $ [0, b] \equiv I_K $ containing the singularity at 0, which has intercept and gradient
$ \alpha = \tfrac{1}{b} \int_{0}^{b} \Phi^{-1}(u) \dd{u} - \tfrac{b\beta}{2}  $
and 
$ \beta  = \tfrac{12}{b^3} \int_{0}^{b} (u - \tfrac{b}{2}) \Phi^{-1}(u) \dd{u} $.
These integrals can be calculated exactly (with a change of variables), where denoting $ z_b \coloneqq \Phi^{-1}(b) $ we obtain
\begin{equation*}
\alpha = \dfrac{2\phi(z_b)}{b} - \dfrac{3\Phi(\sqrt{2}z_b)}{b^2\sqrt{\pi}}
\qquad \text{and} \qquad 
\beta  = \dfrac{6}{b^3} \left(\dfrac{\Phi(\sqrt{2} z_b)}{\sqrt{\pi}} - b\phi(z_b)\right).
\end{equation*}
For the gradient, as the interval becomes ever smaller and $ b \to 0 $, we can use lemma~\ref{lemma:approximate_tail_values} to give
\begin{equation*}
\beta  
\approx -\dfrac{6}{b^3} \left(\phi^2(z_b)\left(\dfrac{1}{z_b} - \dfrac{1}{2 z_b^3}\right) + b \phi(z_b)\right)
\approx -\dfrac{6z_b}{b}\left(\dfrac{1}{2z_b^2} + O(z_b^{-4})\right)
\approx \dfrac{-3}{bz_b},
\end{equation*}
where in the first approximation we used  $ \phi(\sqrt{2}z) \equiv \sqrt{2\pi}\phi^2(z) $. Interestingly, this means the range of values $ \beta b \approx -\tfrac{3}{z_b} \to 0 $, and our approximation ``flattens'' relative to the interval $ [0, b] $ as $ b \to 0 $.

With the intercept and gradient in the singular interval $ I_K $ known exactly, we define the two points $ u_- $ and $ u_+ $ where the error is zero, where $ 0 < u_- < u_+ < b $, and there are two as $ \Phi^{-1} $ is concave in $ (0, \tfrac{1}{2}) $. Corresponding to these we define $ z_- \coloneqq \Phi^{-1}(u_-) $ and $ z_+ \coloneqq \Phi^{-1}(u_+) $, where $ -\infty < z_- < z_+ < z_b < 0 $. Thus in the singular interval we have
\begin{equation*}
\int_{I_K} \lvert D_K(u) - \Phi^{-1}(u)\rvert^p \dd{u} 
= \int_{-\infty}^{z_-} \lvert z - \alpha - \beta \Phi(z)\rvert^p \phi(z) \dd{z} + \int_{z_-}^{z_b} \lvert z - \alpha - \beta \Phi(z)\rvert^p \phi(z) \dd{z}.
\end{equation*}
Using lemmas~\ref{lemma:approximate_tail_values} and \ref{lemma:approximate_moments}, then for the first of these integrals we obtain
\begin{equation*}
\int_{-\infty}^{z_-} \lvert z - \alpha - \beta \Phi(z)\rvert^p \phi(z) \dd{z}
\leq \int_{-\infty}^{z_-} \lvert z - z_-\rvert^p \phi(z) \dd{z}
\leq \int_{-\infty}^{z_b} \lvert z - z_b\rvert^p \phi(z) \dd{z}
\approx  \dfrac{p!\phi(z_b)}{\lvert z_b\rvert^{p+1}}
=  O\left(\dfrac{b}{\lvert z_b \rvert^p}\right),
\end{equation*}
and for the second integral we similarly obtain
\begin{equation*}
\int_{z_-}^{z_b} \lvert z - \alpha - \beta \Phi(z)\rvert^p \phi(z) \dd{z}
\leq \int_{-\infty}^{z_b} \lvert z_+ - \alpha - \beta \Phi(z)\rvert^p \phi(z) \dd{z}
\leq \lvert \beta b \rvert^p \int_{0}^{b}\dd{u}
\approx \dfrac{3^p}{\lvert z_b\rvert^p} b
=  O\left(\dfrac{b}{\lvert z_b \rvert^p}\right).
\end{equation*}

Combining the results for the central intervals and the singular interval we obtain
\begin{equation*}
\lVert D - \Phi^{-1}\rVert_p^p
= O((1-r)^{2p}) + O\left(\dfrac{b}{\lvert z_b \rvert^p}\right) 
= O((1-r)^{2p}) +
O(r^{K-1} ({\log}(r^{1-K}\sqrt{2/\pi}))^{-p/2} )  
= O((1-r)^{2p}) +
o(r^{K-1}),
\end{equation*}
where for the second equality we used lemma~\ref{lemma:approximate_tail_values} in the limit $ r^{K-1} \to 0 $. \qedhere
\end{proof}

We can see from theorem~\ref{thm:piecewise_linear_approximation_error} that the $ O((1 - r)^{2p}) $ term comes from the central regions, and is reduced by taking $ r \to 1 $, and the $ o(r^{K-1}) $ term is from the singular interval, and is reduced by taking $ r^{K-1} \to 0 $. The key point of interest with this result is that the error from the central regions and the singular region is decoupled. In order to decrease the overall error, it is not sufficient to only increase the number of intervals ($ K \to \infty $), which would only improve the error from the singular interval, but the decay rate must also be decreased ($ r \to 1 $). The independence and interplay of these two errors is important for balancing the fidelity between the central and edge regions.

It is possible to generalise this construction and analysis to piecewise polynomial approximations, constructing the best $L^2$ approximation on each interval. We could also require the approximations to be continuous over the entire interval $ (0, 1) $ by turning our previous minimisation into a constrained minimisation, with a coupled set of linear constraints. While such a continuous approximation may be more aesthetically pleasing, it is of no practical consequence for the inverse transform method, and by definition will have a worse overall $ L^2 $ error than the discontinuous approximation from theorem~\ref{thm:piecewise_linear_approximation_error}, and thus we choose not to do this.

The errors from piecewise polynomial approximations for various polynomial orders and interval numbers are shown in figure~\ref{fig:piecewise_linear_gaussian_approximation_error}, where we have set the decay rate $ r = \tfrac{1}{2} $. As we increase the number of intervals used, then for the piecewise linear function, the $ L^2 $ error plateaus to approximately $ 10^{-2} $ for 16 intervals, which is approximately equal to the error for the piecewise constant approximation using $ 10^3 $ intervals. Thus we can appreciate the vast increase in fidelity that the piecewise linear approximation naturally offers over the piecewise constant approximation. Furthermore, this plateau demonstrates that the central regions are limiting the accuracy of the approximation. Inspecting the approximation using 16 intervals in $ (0, \tfrac{1}{2}) $, we see that as we increase the polynomial order from linear to cubic there is a considerable drop in the error approximately equal to a factor of $ 10^2 $, where the piecewise cubic approximation is achieving $ L^2 $ errors of around $ 10^{-4} $. However, we see that increasing the polynomial order any further does not lead to any significant reduction in the $ L^2 $ error, indicating that in this regime it is the error in the singular interval which is dominating the overall error. 

\section{High performance implementations}
\label{sec:high_performance_impementations}

Our motivation for presenting the piecewise constant and linear approximations was to speed up the inverse transform method, where the approximations were inherently simple in their constructions, with the suggestion that any implementations would likely be simple and performant. However, it is not obvious that the implementations from various software libraries, especially heavily optimised commercial libraries such as Intel's Maths Kernel Library (MKL), should be slow. In this section, we will mention how most libraries implement these routines, and why many are inherently ill posed for modern vector hardware. Our two approximations capitalise on the features where most libraries stumble, avoiding division and conditional branching, which will be the key to their success. We give a brief overview of their implementations in C and showcase their superior speed across Intel and Arm hardwares. While we will briefly outline the implementations here in moderate detail, a more comprehensive suite of implementations and experiments, including Intel AVX-512 and Arm SVE specialisations, are hosted in a centralised repository by \citet{sheridan2020approximate_random}.

\subsection{The shortcomings in standard library implementations}

Several libraries providers, both open source and commercial, offer implementations of the inverse Gaussian cumulative distribution function, (Intel, NAG, Nvidia, MathWorks, Cephes, GNU, Boost, etc.). The focus of these implementations is primarily ensuring near machine precision is achieved for all possible valid input values, and handling errors and edge cases appropriately. While some offer single precision implementations, most offer (or default to) double precision, and to achieve these precisions algorithms have improved over the years \citep{hastings1955approximations,evans1974algorithm70,beasley1985percentage,wichura1988algorithm,marsaglia1994rapid,giles2011approximating}. The most widely used is by \citet{wichura1988algorithm}, which is used in GNU's scientific library (GSL) \citep{galassi2017gsl}, the Cephes library \citep{moshier1992cephes} (used by Python's SciPy package \citep{scipy2020scipy}), and the NAG library \citep{nag2017mark26} (NAG uses a slight modification). 

These implementations split the function into two regions: an easy central region, and the difficult tail regions. Inside the central region a rational Pad\'{e} approximation is used with seventh order polynomials \citep{wichura1988algorithm}. In the tail regions, the square root of the logarithm is computed, and then a similar rational approximation performed on this. While these are very precise routines, they have several shortcomings with respect to performance.

The first shortcoming is the division operation involved in computing the rational approximation. Division is considerably more expensive than addition or multiplication, and with a much higher latency \citep{wittmann2015short,fog2018instruction}. Thus, for the very highest performance, avoiding division is preferable. (The implementation by \citet{giles2011approximating} already capitalises on this). Similarly, taking logarithms and square roots is expensive, making the tail regions very expensive, and so are best avoided too. 

The second short coming is that the routine branches: requesting very expensive calculations for infrequent tail values, and less expensive calculations for more frequent central values. Branching is caused by using ``\texttt{if-else}'' conditions, but unfortunately this often inhibits vectorisation, resulting in compilers being unable to issue single instruction multiple data (SIMD) operations \citep{vanderpas2017using}. Even if a vectorised SIMD implementation is produced, typically the results from both branches are computed, and the correct result selected \textit{a posteriori} by predication/masking. Thus, irrespective of the input, both the very expensive and less expensive calculations may be performed, resulting in the relatively infrequent tail values producing a disproportionately expensive overall calculation. This effect is commonly known as \emph{warp divergence} in GPU programs, and in our setting we will call this \emph{vector divergence}. The wider the vector and the smaller the data type, (and thus the greater the amount of vector parallelisation), the worse the impact of vector divergence. Noting that half precision uses only 16 bits, both Intel's AVX-512 and Fujitsu's A64FX are 512 bits wide, and Arm's SVE vectors can be up to 2048 bits \citep{petrogalli2016sneak_peak,stephens2017arm}, the degree of vector parallelisation can be substantial, and the impact of conditional branching becomes evermore crippling as vector parallelisation increases. 

In light of these two shortcomings, we will see our approximations can be implemented in a vectorisation friendly manner which avoids conditional branching and are homogenous in their calculations. Furthermore, their piecewise constant and polynomial designs avoid division operations, and only require addition, multiplication, and integer bit manipulations, and thus result in extremely fast executables. We will see that our piecewise constant approximation will rely on the high speed of querying the cache. Furthermore, for the piecewise linear approximation using a decay rate $ r = \tfrac{1}{2} $ and 16 intervals, then in single precision all the coefficients can be held in 512 bit wide vector registers, bypassing the need to even query the cache, and thus are extremely fast provided the vector widths are sufficiently large. 

\subsection{Implementing the piecewise constant approximation}

Splitting the domain $ (0, 1) $ into the $ N $ intervals $ [\tfrac{m}{N}, \tfrac{m+1}{N}) $ zero indexed by $ m \in \{0,1,2,\ldots,N-1\}$, the values for each interval are easily computed \textit{a priori}, and stored in a lookup table. An example of how this looks in C is shown in code~\ref{code:piecewise_constant_approximation}, where we use OpenMP (\texttt{omp.h}) to signal to the compiler that the \texttt{for} loop is suitable for vectorisation. This relies on typecasting to an integer, where any fractional part is removed. The benefit to such an implementation is that on modern chips a copy of the lookup table can readily fit within either the L1 or L2 caches, where 1024 values stored in 64 bit double precision consume \SI{8}{\kilo\byte}. As an example, an Intel Skylake Xeon Gold 6140 CPU has L1 and L2 caches which are \SI{32}{\kilo\byte} and \SI{2}{\mega\byte} respectively, and thus the lookup table can exploit the fast speeds of the L1 cache, which typically has latencies of 2--5 clock cycles. 

\begin{lstfloat}[htb]
\begin{lstlisting}[style=C, caption={C implementation of the piecewise constant approximation.}, label={code:piecewise_constant_approximation}]
#define LOOKUP_TABLE_SIZE 1024
const float lookup_table[LOOKUP_TABLE_SIZE] = {-3.3, -2.9, -2.8, ..., 2.8, 2.9, 3.3};

void piecewise_constant_approximation(const unsigned int n_samples, 
                                      const float * restrict input, 
                                      float * restrict output) {
    #pragma omp simd
    for (unsigned int n = 0; n < n_samples; n++) 
        output[n] = lookup_table[(unsigned int) (LOOKUP_TABLE_SIZE * input[n])];
}
\end{lstlisting}
\end{lstfloat}

\subsection{Implementing the piecewise linear approximation}

The piecewise linear approximation is a bit more involved than the piecewise constant approximation. Not only do we have a polynomial to evaluate (albeit only linear), but the varying widths of the intervals means identifying which interval a given value corresponds to is more involved. Once the appropriate interval's polynomial's coefficients are found, evaluating the polynomial is trivial, where the linear polynomial can be evaluated using a fused multiply and add (FMA) instruction. Higher order polynomials can be similarly computed (e.g.\ with Horner's rule). 

The primary challenge is determining which interval a given value corresponds to based on the construction from theorem~\ref{thm:piecewise_linear_approximation_error} for a given decay rate $ r $. As floating point numbers are stored in binary using their sign, exponent, and mantissa, the natural choice most amenable for computation is $ r = \tfrac{1}{2} $, which we call the \emph{dyadic} rate, producing the dyadic intervals shown in table~\ref{tab:dyadic_intervals}. From table~\ref{tab:dyadic_intervals}, we notice the intervals are only dense near the singularity at 0, but not at 1. This is not problematic, as we said in theorem~\ref{thm:piecewise_linear_approximation_error} that the approximation is anti-symmetric about $ u=1/2 $, and thus we can use the $ N $ intervals in $ (0, \tfrac{1}{2}) $, and if our input value is within $ (\tfrac{1}{2}, 1) $, it is straight forward to simply negate the value computed for the input reflected about $ \tfrac{1}{2} $ (equivalent to using the inverse complementary cumulative distribution function). 

\begin{table}[htb]
\centering
\caption{Dyadic intervals and their corresponding array indices.}
\label{tab:dyadic_intervals}
\renewcommand{\arraystretch}{1.4}  
\begin{tabular}{c|ccccccc}
Index & 0 & 1 & 2 & $ \cdots $ & $ n $ & $ \cdots $ & $ N $ \\ \hline
Interval & $ [\tfrac{1}{2}, 1) $ & $ [\tfrac{1}{4}, \tfrac{1}{2}) $ & $ [\tfrac{1}{8}, \tfrac{1}{4}) $ & $ \cdots $ & $ [\tfrac{1}{2^{n+1}}, \tfrac{1}{2^n}) $ & $ \cdots $ & $ (0, \tfrac{1}{2^N}) $
\end{tabular}
\end{table}

Such an implementation can handle any value in $ (0, \tfrac{1}{2}) \cup (\tfrac{1}{2}, 1) $, but unfortunately $ \tfrac{1}{2} $ is not included within this. While a mathematician may say the individual value $ \tfrac{1}{2} $ has zero measure, from a computational perspective such a value is perfectly feasible input, and quite likely to appear in any software tests. Thus our implementation will correctly handle this value. The reader will notice that when we reflect an input value $ x > \tfrac{1}{2} $ about $ \tfrac{1}{2} $ by using $ x \to 1 - x $, then the value $ \tfrac{1}{2} $ will not be in any of the intervals indexed between 1 and $ N $ in table~\ref{tab:dyadic_intervals}, but remains in the interval indexed by 0. Thus in our later implementations, the arrays of coefficients will always hold as their first entry (index 0) zero values ($ \Phi^{-1}(\tfrac{1}{2}) = 0 $) so $ \tfrac{1}{2} $ is correctly handled. We found during the development of the implementation, being able to correctly handle this value is of considerable practical importance, so we encourage practitioners to also correctly handle this value. 

If we are using the dyadic decay rate, then the interval an input $ x $  belongs to is $ \lceil -{\log}_2(x) \rceil - 1 $. However, to compute this we need not take any logarithms, which are expensive to compute. We can obtain the same result by reading off the exponent bits in the floating point representation, treating these as an integer, and then correcting for the exponent bias.  This only involves simple bit manipulations, and interpreting the bits in a floating point representation as an integer. In C this is a technique called type punning, and can be achieved either by pointer aliasing, using a \texttt{union}, or using specialised intrinsic functions. Of these, pointer aliasing technically breaks the strict aliasing rule in C (and C++) \citep[6.5.2.3]{iso2012c11} \citep[pages~163--164]{stallman2020gcc}. Similarly, type punning with a \texttt{union} in C89 is implementation defined, whereas in C11 the bits are re-interpreted as desired. In the presented implementation we will leave this detail undefined and just use a macro to indicate the type punning. (The implementation approximating the Gaussian distribution uses pointer aliasing, whereas the later approximation to the non central $ \chi^2 $ distribution in section~\ref{sec:the_non_central_chi_squared_distribution} uses a \texttt{union}, to demonstrate both possibilities \citep{sheridan2020approximate_random}).

\begin{algorithm}[h!tb]
\DontPrintSemicolon
\KwIn{Floating point uniform random variable $ U \in (0, 1) $.}
\KwOut{Floating point approximate Gaussian random variable $ \widetilde{Z} $.}
Form predicate using $ U > \tfrac{1}{2} $.\;
Reflect about $ \tfrac{1}{2} $ to obtain $ U \in (0, \tfrac{1}{2}] $.\;
Interpret $ U $ as an unsigned integer.\;
Read the exponent bits using bit wise AND.\;
Right shift away the mantissa's bits.\;
Obtain an array index by correcting for exponent bias.\;
Cap the array index to avoid overflow.\;
Read the polynomial's coefficients.\;
Re-interpret $ U $ as a float.\;
Form the polynomial approximation $ \widetilde{Z} $.\;
Correct sign of approximation based on the predicate.\;
\caption{Piecewise polynomial approximation using dyadic intervals.}
\label{algo:piecewise_polynomial_approximation_using_dyadic_intervals}
\end{algorithm}

Overall then, the general construction from theorem~\ref{thm:piecewise_linear_approximation_error} is given in algorithm~\ref{algo:piecewise_polynomial_approximation_using_dyadic_intervals}. Furthermore, a single precision C implementation is shown in code~\ref{code:piecewise_linear_approximation}, where we assume the single precision floats are stored in IEEE 32 bit floating point format \citep{ieee2008ieee}. 

\begin{lstfloat}[h!tb]
\begin{lstlisting}[style=C, caption={C implementation of the piecewise linear approximation.}, label={code:piecewise_linear_approximation}]
typedef unsigned int uint32;  // Assuming 32 bit floats. 
typedef float float32;        // Assuming 32 bit integers. 

#define TABLE_SIZE 16
#define N_MANTISSA_32 23            // For IEEE 754
#define FLOAT32_EXPONENT_BIAS 127   // For IEEE 754
#define FLOAT32_EXPONENT_BIAS_TABLE_OFFSET (FLOAT32_EXPONENT_BIAS - 1)
#define TABLE_MAX_INDEX (TABLE_SIZE - 1) // Zero indexing.
#define FLOAT32_AS_UINT32(x) (...)       // Type punning.

const float32 poly_coef_0[TABLE_SIZE] = {0.0, -1.3, -1.6, ..., -4.0, -4.1, -4.5};
const float32 poly_coef_1[TABLE_SIZE] = {0.0, 2.6, 3.7, ..., 2800.0, 5300.0, 21000.0};

#pragma omp declare simd
static inline float32 polynomial_approximation(const float32 u, const uint32 index) {
    return poly_coef_0[index] + poly_coef_1[index] * u;
}

#pragma omp declare simd
static inline uint32 get_table_index_from_float_format(const float32 u) {
    uint32 index = FLOAT32_AS_UINT32(u) >> N_MANTISSA_32;      // Remove the mantissa.
    index = FLOAT32_EXPONENT_BIAS_TABLE_OFFSET - index;        // Get the index.
    return index > TABLE_MAX_INDEX ? TABLE_MAX_INDEX : index;  // Avoid overflow.
}

void piecewise_polynomial_approximation(const unsigned int n_samples,
                                        const float32 * restrict input, 
                                        float32 * restrict output) {
    #pragma omp simd 
    for (unsigned int i = 0; i < n_samples; i++) {
        float32 u = input[i];
        bool predicate = u < 0.5;
        u = predicate ? u : 1.0 - u;
        uint32 index = get_table_index_from_float_format(u);
        float32 z = polynomial_approximation(u, index);
        z = predicate ? z : -z;
        output[i] = z;
    }
}
\end{lstlisting}
\end{lstfloat}

The reason why we decide to implement the approximation in single precision using coefficient arrays of 16 entries is because each requires only 512 bits ($ 16 \times 32 $ bits) to store all the possible values for a given monomial's coefficient. The significance of 512 bits cannot be overstated, as it is the width of an AVX-512 and A64FX vector register. Thus, instead of querying the cache to retrieve the coefficients, they can be held in vector registers, bypassing the cache entirely, and achieving extremely fast speeds. Recognising the coefficients can be stored in a single coalesced vector register is currently largely beyond most compilers' capabilities using OpenMP directives and compiler flags alone. However, in the repository by \citet{sheridan2020approximate_random}, specialised implementations using Intel vector intrinsics and Arm inline assembly code achieve this, obtaining the ultimate in performance. 

\subsection{Performance of the implementations}

Both the piecewise constant and linear implementations in codes~\ref{code:piecewise_constant_approximation} and \ref{code:piecewise_linear_approximation} are non branching, vector capable, and use only basic arithmetic and bit wise operations, and thus we anticipate their performance should be exceptionally good. Indeed, their performance, along with several other implementations', is shown in table~\ref{tab:implementation_times}, with experiments performed on an Intel Skylake Xeon Gold CPU and an Arm based Cavium ThunderX2 \citep{sheridan2020approximate_random}.

\begin{table}[htb]
\centering
\caption{Performance of various approximations and implementations of the inverse Gaussian cumulative distribution function, with performance measured as clock cycles per random number.}
\label{tab:implementation_times}
\begin{tabular}{lccccc}
Description & Implementation & Hardware & Compiler & Precision &  Clock cycles\\ 
\hline
Cephes  \citep{moshier1992cephes} & --- &  Intel & \texttt{icc} & Double & $ 60 \pm 1 $ \\
GNU GSL & --- &  Intel & \texttt{icc} & Double & $ 52 \pm 10 $ \\
ASA241  \citep{wichura1988algorithm,burkardt2020software} & --- &  Intel & \texttt{icc} & Single & $ 47 \pm 1 $ \\
\citet{giles2011approximating} & --- & Intel & \texttt{icc} & Single & $ 46 \pm 2 $ \\
Intel (HA) &  MKL VSL &  Intel & \texttt{icc} & Double &  $ 9 \pm 0.5 $ \\
Intel (LA)   & MKL VSL &  Intel & \texttt{icc} & Double &   $ 7 \pm 0.5 $ \\
Intel (HA) &  MKL VSL &  Intel & \texttt{icc} & Single &  $ 3.4 \pm 0.1 $ \\
Intel (LA)   & MKL VSL &  Intel & \texttt{icc} & Single &   $ 2.6 \pm 0.1 $ \\
Piecewise constant  & OpenMP & Arm & \texttt{armclang} & Double & $ 4.0 \pm 0.5 $ \\
Piecewise constant  & OpenMP & Intel & \texttt{icc} & Double & $ 1.5 \pm 0.3 $ \\
Piecewise cubic  & OpenMP &   Intel & \texttt{icc} & Single &   $ 0.9 \pm 0.1 $  \\
Piecewise cubic    &  Intrinsics  &   Intel & \texttt{icc} & Single &  $ 0.7 \pm 0.1 $ \\
Piecewise linear&  Intrinsics  &   Intel & \texttt{icc} & Single &   $ 0.5  \pm 0.1 $ \\
Read and write & --- & Intel & $ \texttt{icc} $ & Single & $ 0.4 \pm 0.1 $ 
\end{tabular}
\end{table}

Looking at the results from table~\ref{tab:implementation_times}, of all of these, the piecewise linear implementation using Intel vector intrinsics achieves the very fastest speeds, closely approaching the maximum speed of just reading and writing. Unsurprisingly, the freely available implementations from Cephes and GSL are not competitive with the commercial offerings from Intel. Nonetheless, even in single precision, our approximations, on Intel hardware, consistently beat the performance achieved from the Intel high accuracy (HA) or low accuracy (LA) offerings. Comparing the high accuracy Intel offering and the piecewise linear implementation, there stands to be a speed up by a factor of seven by switching to our approximation. These results vindicate our efforts, and that our simple approximations offer considerable speed improvements. It is also needless to say, that compared to the freely available open source offerings, the savings become even more significant. 

\section{Multilevel Monte Carlo}
\label{sec:multilevel_monte_carlo}

One of the core use cases for our high speed approximate random variables is in Monte Carlo applications. Frequently, Monte Carlo is used to estimate expectations of the form $ \mathbb{E}(P) $ of functionals $ P $ which act on solutions $ X $ of stochastic differential equations of the form $ \dd{X_t} = a(t, X_t) \dd{t} + b(t, X_t)\dd{W_t} $ for given drift and diffusion processes $ a $ and $ b $. The underlying stochastic process $ X_t $ is itself usually approximated by some $ \widehat{X} $ using a numerical method, such as the Euler-Maruyama or Milstein schemes \citep{asmussen2007stochastic,kloeden1999numerical,lord2014introduction}. These approximation schemes simulate the stochastic process from time $ t = 0 $ to $ t = T $ over $ N $ time steps of size $ \Delta t = \delta = \tfrac{T}{N} $, where the incremental update at the $ n $-th iteration requires a Gaussian random variable $ Z_n $ to simulate the increment to the underlying Wiener process $ W_t $, where $ \Delta W_n = \sqrt{\delta}Z_n $. Such types of Monte Carlo simulations are widespread, with the most famous application being to price financial options. 

Approximate random variables come into this picture by substituting the exact random variable samples in the numerical scheme with approximate ones. This facilitates running faster simulations, at the detriment of introducing error. However, using the multilevel Monte Carlo method \citep{giles2008multilevel}, this error can be compensated for with negligible cost. Thus, the speed improvements offered by switching to approximate random variables can be largely recovered, and the original accuracy can be maintained. A detailed inspection of the error introduced from incorporating approximate Gaussian random variables into the Euler-Maruyama scheme and the associated multilevel Monte Carlo analysis is presented by \citeauthor{giles2020approximate} \citep{giles2020approximate,sheridan2020nested}. As such, we will only briefly review the key points of the setup, and focus on detailing the resultant computational savings that can be expected from using approximate random variables. 

For the Euler-Maruyama scheme, the unmodified version using exact Gaussian random variables $ Z $ produces an approximation $ \widehat{X} $, whereas the modified scheme using approximate random variables $ \widetilde{Z} $ produces an approximation $ \widetilde{X} $, where the two schemes are respectively
\begin{equation*}
\widehat{X}_{n+1} = \widehat{X}_n + a(t_n, \widehat{X}_n) \delta + b(t_n, \widehat{X}_n)\sqrt{\delta} Z_n
\qquad \text{and} \qquad 
\widetilde{X}_{n+1} = \widetilde{X}_n + a(t_n, \widetilde{X}_n) \delta + b(t_n, \widetilde{X}_n)\sqrt{\delta} \widetilde{Z}_n,
\end{equation*}
where $ t_n \coloneqq n \delta $. 

The regular multilevel Monte Carlo construction varies the discretisation between two levels, producing \emph{fine} and \emph{coarse} simulations. The functional $ P $ would act on each path simulation, producing the fine and coarse approximations $ P^{\mathrm{f}} $ and $ P^{\mathrm{c}} $ respectively. If the path simulation uses the exact random variables we denote these as $ \widehat{P}^{\mathrm{f}} $ and $ \widehat{P}^{\mathrm{c}} $, and alternatively if it uses approximate random variables as $ \widetilde{P}^{\mathrm{f}} $ and $ \widetilde{P}^{\mathrm{c}} $. In general there may be multiple tiers of fine and coarse levels, so we use $ l $ to index these, where increasing values of $ l $ correspond to finer path simulations. Thus for a given $ l $ we have $ \widehat{P}_l \equiv \widehat{P}^{\mathrm{f}} $ and $ \widehat{P}_{l-1} \equiv \widehat{P}^{\mathrm{c}} $, and similarly $ \widetilde{P}_l \equiv \widetilde{P}^{\mathrm{f}} $ and $ \widetilde{P}_{l-1} \equiv \widetilde{P}^{\mathrm{c}} $. If we have levels $ l \in \{0, 1, 2, \ldots, L\} $ and use the convention $ \widehat{P}_{-1} \coloneqq \widetilde{P}_{-1} \coloneqq 0 $, then \citet{giles2020approximate} suggest the nested multilevel Monte Carlo
\begin{equation*}
E(P) 
\approx
E(\widehat{P}_L) 
= 
\sum_{l = 0}^{L} E(\widehat{P}_l - \widehat{P}_{l-1}) 
= 
\sum_{l = 0}^{L} E(\widetilde{P}_l - \widetilde{P}_{l-1}) +  E(\widehat{P}_l - \widehat{P}_{l-1} - \widetilde{P}_l + \widetilde{P}_{l-1}),
\end{equation*}
where the first approximation is the regular Monte Carlo procedure \citep{glasserman2013monte}, the first equality is the usual multilevel Monte Carlo decomposition \citep{giles2008multilevel}, and the final equality is the nested multilevel Monte Carlo framework \citep{giles2020approximate,sheridan2020nested}.
\citeauthor{giles2020approximate} \citep{giles2020approximate,sheridan2020nested} show that the two way differences in the regular and nested multilevel Monte Carlo settings behave almost identically, and the main result of their analysis is determining the behaviour of the final four way difference's variance \citep[lemmas~4.10 and 4.11]{giles2020approximate} \citep[corollaries~6.2.6.2 and 6.2.6.3]{sheridan2020nested}. They find that for Lipschitz continuous and differentiable functionals
\begin{equation*}
\lVert \widehat{P}^{\mathrm{f}} - \widehat{P}^{\mathrm{c}} - \widetilde{P}^{\mathrm{f}} + \widetilde{P}^{\mathrm{c}}\rVert_p 
\leq O(\delta^{1/2} \lVert Z - \widetilde{Z} \rVert_{p'}) 
\end{equation*}
for any  $ p' $ such that $ 2 \leq p < p' < \infty $, and that for Lipschitz continuous but non-differentiable functionals that 
\begin{equation*}
\lVert \widehat{P}^{\mathrm{f}} - \widehat{P}^{\mathrm{c}} - \widetilde{P}^{\mathrm{f}} + \widetilde{P}^{\mathrm{c}}\rVert_p 
\leq O(\min\{
\delta^{1/2} \lVert Z - \widetilde{Z} \rVert_{p'}^{p'(1-\epsilon)/(p'+1)},
\delta^{(1-\epsilon)/2p -1/2p'}    \lVert Z - \widetilde{Z} \rVert_{p'}
\})
\end{equation*}
for any  $ \epsilon > 0 $. We can see that in all circumstances covered by their analysis that there is a dependence on the approximation error $  \lVert Z - \widetilde{Z} \rVert_{p'} $ for some $ L^{p'} $ norm where $ p' > 2 $.

\subsection{Expected time savings}

The regular multilevel estimator $ \hat{\theta} $ is
\begin{equation*}
  \hat{\theta}  \coloneqq \sum_{l=0}^{L} \dfrac{1}{\widehat{m}_l} \sum_{i=1}^{\widehat{m}_l} \left( \widehat{P}^{(i,l)}_l - \widehat{P}^{(i,l)}_{l-1}\right),
\end{equation*}
and the nested multilevel estimator $ \tilde{\theta} $ is
\begin{equation*}
\tilde{\theta} \coloneqq \sum_{l=0}^L \dfrac{1}{\widetilde{m}_l} \sum_{i=1}^{\widetilde{m}_l}\left( \widetilde{P}^{(i,l)}_l - \widetilde{P}^{(i,l)}_{l-1}\right) + \dfrac{1}{\widetilde{M}_l} \sum_{i=1}^{\widetilde{M}_l} \left( \widehat{P}^{(i,l+L+1)}_l - \widehat{P}^{(i,l+L+1)}_{l-1} - \widetilde{P}^{(i,l+L+1)}_l + \widetilde{P}^{(i,l+L+1)}_{l-1}\right),
\end{equation*}
where $ \widehat{m}_l $, $ \widetilde{m}_l $, and $ \widetilde{M}_l $ are the number of paths generated, each with a computational time cost of $ \hat{c}_l $, $ \tilde{c}_l $, and $ \widetilde{C}_l $, and variance $ \hat{v}_l $, $ \tilde{v}_l $, and $ \widetilde{V}_l $ respectively, and all terms with the same suffix $(i,l)$ are computed using the same uniform random numbers. Each of these estimators will have an error due to the finite number of paths used and the approximation scheme employed. The total error arising from these two factors is commonly referred to as the variance bias trade-off, where the mean squared error (MSE) of an estimator $ \theta \in \{\hat{\theta}, \tilde{\theta}\} $ is given by $ \text{MSE}(\theta) = \mathbb{V}(\theta) + \text{Bias}^2(\theta)$ \citep[page~16]{glasserman2013monte}. Setting the desired MSE to $ \varepsilon^2 $ and choosing the maximum simulation fidelity such that we just satisfy $ \text{Bias}^2(\theta) \leq \tfrac{\varepsilon^2}{2} $, then we can derive an expression for the total computational time $ T $. Forcing $ T $ to be minimal is achieved by performing a constrained minimisation with an objective function $ \mathscr{F} $. Considering the estimator $ \hat{\theta} $, the corresponding objective function is $ \widehat{\mathscr{F}} \coloneqq \sum_{l=0}^{L} \widehat{m}_l \hat{c}_l + \mu (\sum_{l=0}^{L} \tfrac{\hat{v}_l}{\widehat{m}_l} - \tfrac{\varepsilon^2}{2}) $, where $ \mu $ is a Lagrange multiplier enforcing the constraint $ \mathbb{V}(\hat{\theta}) = \tfrac{\varepsilon^2}{2} $. Treating the number of paths as a continuous variable this is readily minimised to give
\begin{equation*}
\widehat{T} = 2\varepsilon^{-2}\left(\sum_{l=0}^L \sqrt{\hat{v}_l \hat{c}_l}\right)^2 
\qquad \text{and} \qquad 
\widetilde{T} = 2\varepsilon^{-2} \left(\sum_{l=0}^L \sqrt{\tilde{v}_l \tilde{c}_l} + \sqrt{\widetilde{V}_l \widetilde{C}_l}\right)^2,
\end{equation*}
where the minimal number of paths required are given by
\begin{equation*}
\widehat{m}_l = \varepsilon^{-1} \sqrt{\dfrac{2\widehat{T}\hat{v}_l}{\hat{c}_l}},
\qquad 
\widetilde{m}_l  = \varepsilon^{-1} \sqrt{\dfrac{2\widetilde{T} \tilde{v}_l}{\tilde{c}_l}}, 
\qquad 
\text{and}
\qquad 
\widetilde{M}_l  = \varepsilon^{-1} \sqrt{\dfrac{2\widetilde{T} \widetilde{V}_l}{\widetilde{C}_l}},
\end{equation*}
and hence an overall saving of
\begin{equation*}
\widetilde{T} 
\approx 2\varepsilon^{-2} \left(\sum_{l=0}^L \sqrt{\hat{v}_l \hat{c}_l} \left( \sqrt{\dfrac{\tilde{v}_l\tilde{c}_l}{\hat{v}_l\hat{c}_l}} + \sqrt{\dfrac{\widetilde{V}_l \widetilde{C}_l}{\hat{v}_l \hat{c}_l}}\right)\right)^2 
\leq \widehat{T} \max_{l \leq L} \left\{ \dfrac{\tilde{v}_l\tilde{c}_l}{\hat{v}_l\hat{c}_l} \left(1 + \sqrt{\dfrac{\widetilde{V}_l \widetilde{C}_l}{\tilde{v}_l \tilde{c}_l}}\right)^2\right\}.
\end{equation*}
For modest fidelity approximations where $ \tilde{v}_l \approx \hat{v}_l $, the term $ \tfrac{\tilde{v}_l\tilde{c}_l}{\hat{v}_l\hat{c}_l} \approx \tfrac{\tilde{c}_l}{\hat{c}_l}$ measures the potential time savings, and the term $ (1 + (\widetilde{V}_l \widetilde{C}_l / \tilde{v}_l \tilde{c}_l)^{1/2})^2 $ assesses the efficiency of realising these savings. A balance needs to be achieved, and the approximations should be sufficiently fast so there is the potential for large savings, but of a sufficient fidelity so the variance of the expensive four way difference is considerably lower than the variance of the cheaper two way difference.

Although we have estimates for the costs of each estimator from table~\ref{tab:implementation_times}, the variance will depend on the stochastic process being simulated, the numerical method being used, and the approximation employed. To make our estimates more concrete and quantify the possible variance reductions, we consider a geometric Brownian motion where $ a(t, X_t) \coloneqq \mu X_t $ and $ b(t, X_t) \coloneqq \sigma X_t $ for two strictly positive constants $ \mu $ and $ \sigma $, where we take $ \mu = 0.05 $, $ \sigma = 0.2 $, $ X_0 = 1 $, and $ T = 1 $ \citep[6.1]{giles2008multilevel}. The coarse level's Weiner increments are formed by pair wise summing the fine level's (which uses the fine time increment $ \delta^{\mathrm{f}} $). Importantly, the uniform random variable samples producing the exact Gaussian random variables will be the same as those used for producing the approximate Gaussian random variables, ensuring a tight coupling. Thus, for some approximation $ \widetilde{\Phi}^{-1} \approx \Phi^{-1} $, such as those from section~\ref{sec:approximate_gaussian_random_variables}, for a single uniform sample $ U_n \sim \mathcal{U}(0, 1) $, then $ Z_n \coloneqq \Phi^{-1}(U_n) $ and $ \widetilde{Z}_n \coloneqq \widetilde{\Phi}^{-1}(U_n) $, where the $ U_n $ is the same for both of these.

The variance for the various multilevel terms for different time increments for the Euler-Maruyama and Milstein schemes are shown in figure~\ref{fig:variance_reduction} for various approximations. We consider the underlying process itself, corresponding to the functional $ P(X) = X $. The piecewise constant approximation uses 1024 intervals, which from figure~\ref{fig:piecewise_constant_gaussian_approximation_error} has an RMSE of approximately $ 10^{-2} $. The piecewise linear approximation uses 15 dyadic intervals (16 coefficients stored to correctly handle $ \tfrac{1}{2} $), which from figure~\ref{fig:piecewise_linear_gaussian_approximation_error} has a similar RMSE of approximately $ 10^{-2} $. The piecewise cubic approximation uses the same number of intervals, which from figure~\ref{fig:piecewise_linear_gaussian_approximation_error} has an RMSE of approximately $ 10^{-4} $. We also include the approximation using Rademacher random variables for comparison.

\begin{figure}[htb]
\centering

\hfil
\subfigure[The Euler-Maruyama scheme.\label{fig:variance_reduction_euler_maruyama_scheme}]{\includegraphics{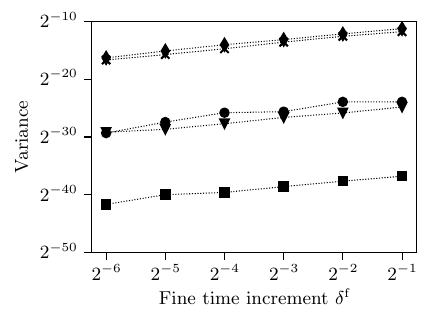}}\hfil 
\subfigure[The Milstein scheme.\label{fig:variance_reduction_milstein_scheme}]{\includegraphics{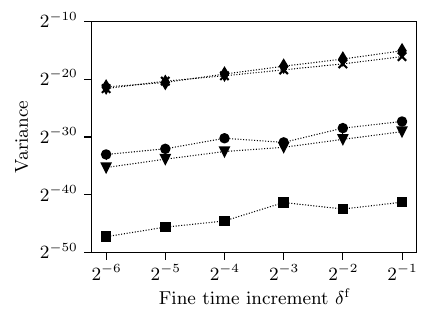}}\hfil

\caption{The variances from using the Euler-Maruyama and Milstein schemes for a geometric Brownian motion for the functional $ P(X) = X $. The ($ \blacklozenge $)-marker is the two way difference $ \widehat{P}^{\mathrm{f}} - \widehat{P}^{\mathrm{c}} $, and the remaining markers ($ \bm{\times} $, \raisebox{-0.1em}{\huge$ \bullet$}, {\large $ \blacktriangledown $}, $ \blacksquare $) are the four way difference $ \widehat{P}^{\mathrm{f}} - \widehat{P}^{\mathrm{c}} - \widetilde{P}^{\mathrm{f}} + \widetilde{P}^{\mathrm{c}} $ for various approximations. ($ \bm{\times} $) Rademacher. (\raisebox{-0.1em}{\huge$ \bullet$}) Piecewise constant. ({\large $ \blacktriangledown $}) Piecewise linear. ($ \blacksquare $) Piecewise cubic.}
\label{fig:variance_reduction}

\end{figure}

We can see from figure~\ref{fig:variance_reduction} that the two way differences exhibit the usual strong convergence order $ \tfrac{1}{2} $ and 1 of the Euler-Maruyama and Milstein schemes, as expected \citep{kloeden1999numerical}. Furthermore, as the functional is differentiable and Lipschitz continuous, this strong convergence rate is preserved for the four way differences \citep{giles2020approximate,sheridan2020nested}. \citeauthor{giles2020approximate} \citep{giles2020approximate,sheridan2020nested} derive this result for the Euler-Maruyama scheme, but the analysis for the Milstein scheme remains an open problem. Aside from this, the key point to note from figure~\ref{fig:variance_reduction} is the substantial drop in the variance between the two way and four way variances using our approximations. This drop in the variance is large for the piecewise constant and linear approximations, and huge for the piecewise cubic approximation. The approximation using Rademacher random variables has only a very small drop in variance. 

Estimating the best cost savings seen compared to Intel (HA) from table~\ref{tab:implementation_times}, the variance reductions from figure~\ref{fig:variance_reduction_euler_maruyama_scheme}, and using the simplifications $ \widetilde{C}_l \approx \hat{c}_l + \tilde{c}_l $ and $ \tilde{v}_l \approx  \hat{v}_l $, then the estimated speed ups and their efficiencies are shown in table~\ref{tab:savings}. For the Rademacher random variables we use the optimistic  approximations that this achieves the maximum speed set by reading and writing and that it offers a variance reduction of $ 2^{-1} $. We can see from this that the piecewise linear approximation would give the largest savings, although the savings from the piecewise constant and cubic approximations are quite similar. Notice that while the piecewise cubic is achieving near perfect efficiency, its cost savings are not substantial enough to beat the marginally less efficient piecewise linear approximation which offers the larger potential savings. For all of our approximations the four way difference simulations are required very infrequently. Lastly, while the Rademacher random variables may offer the best time savings, the absence of any considerable variance reduction caused by the extremely low fidelity approximation results in a very inefficient multilevel scheme, demonstrating the balance required between speed and fidelity. 

\begin{table}[htb]
\centering
\caption{The cost savings, variance reductions, and possible speed ups from approximating the Gaussian distribution.}
\label{tab:savings}
\begin{tabular}{lcclrcr@{}l}
Approximation  & $ {\log}_2 \left(\tfrac{\widetilde{V}_l}{\hat{v}_l}\right) $ & $ \tfrac{\hat{c}_l}{\tilde{c}_l} $ & \multicolumn{2}{c}{Speed up} & $ \tfrac{\widetilde{m}_l}{\widehat{m}_l} $ & \multicolumn{2}{c}{$ \tfrac{\widetilde{m}_l}{\widetilde{M}_l} $} \\[0.5em]
\hline
Rademacher & $ -1 $ & 9 & 0.86 & (9.5\%) & 3.24 & 1&.4 \\
Piecewise constant & $ -13 $ & 6 & 5.66 & (94.4\%) & 1.03 & 240& \\
Piecewise linear  & $ -14 $ & 7& 6.70 & (95.7\%) & 1.02 & 360 &\\
Piecewise cubic  & $ -25 $ & 5 & 5.00 & (99.9\%)& 1.00 & 14000 & 
\end{tabular}
\end{table}

These cost savings are idealised in the respect that we have attributed the cost entirely to the generation of the random numbers. While this is quite a common assumption, the validity of this assumption will diminish the faster the approximations become, as the basic cost of the other arithmetic operations becomes significant. Thus, while a practitioner should have their ambitions set to achieve these savings, they should set more modest expectations.

\section{The non central \texorpdfstring{$ \bm{\chi^2} $}{chi-squared} distribution}
\label{sec:the_non_central_chi_squared_distribution}

A second distribution of considerable practical and theoretical interest is the non-central $ \chi^2 $ distribution, which regularly arises from the Cox-Ingersoll-Ross (CIR) interest rate model \citep{cox1985theory} (and the Heston model \citep{heston1993closed}). The distribution is parametrised as $ \chi^2_\nu(\lambda) $, where $ \nu > 0 $ denotes the degrees of freedom and $ \lambda \geq 0 $ the non-centrality parameter, where we denote the inverse cumulative distribution function as $ C^{-1}_\nu(\cdot; \lambda) $. Having a parametrised distribution naturally increases the complexity of any implementation, exact or approximate, and makes the distribution considerably more expensive to compute. 

To gauge the procedure used, the associated function in Python's SciPy package (\texttt{ncx2.ppf}) calls the Fortran routine \texttt{CDFCHN} from the CDFLIB library by \citet{brown1994dcdflib} (C implementations available \citep{burkardt2020cdflib}). This computes the value by root finding  \citep[algorithm~R]{bus1975two} on the offset cumulative distribution function $ C_\nu(\cdot;\lambda) $, where $ C_\nu(\cdot;\lambda) $ is itself computed by a complicated series expansion \citep[(26.4.25)]{abramowitz1948handbook} involving the cumulative distribution function for the central $ \chi^2 $ distribution. Overall, there are many stages involved, and as remarked by \citet[\texttt{cdflib.c}]{burkardt2020cdflib}: \textit{``Very large values of
[$ \lambda $] can consume immense computer resources''}. The analogous function \texttt{ncx2inv} in MATLAB, from its statistics and machine learning toolbox, appears to follow a similar approach based on its description \citep[page~4301]{matlab2018statistics}. To indicate the costs, on an Intel core i7-4870HQ CPU the ratios between sampling from the non-central $ \chi^2 $ distribution and the Gaussian distribution (\texttt{norm.ppf} and \texttt{norminv} in Python and MATLAB respectively) are shown in table~\ref{tab:non_central_chi_2_times}, from which it is clear that the non-central $ \chi^2 $ distribution can be vastly more expensive than the Gaussian distribution. 

\begin{table}[h!tb]
\centering    
\caption{The computing ratio between the non-central $ \chi^2 $ and Gaussian distributions.}
\label{tab:non_central_chi_2_times}

\hfil
\subfigure[Python.\label{tab:non_central_chi_2_times_python}]{
\begin{tabular}{|r|rrrrr|}
\multicolumn{1}{c}{\multirow{2}{*}{$ \lambda $}} & \multicolumn{5}{c}{$ \nu $} \\
\cline{2-6}
\multicolumn{1}{c|}{} & 1 &   5  &  10 &  50  & 100 \\
\hline
1    &  37 &  36 &  40 &  54 &  73\\
5    &  40 &  46 &  48 &  62 &  85\\
10   &  54 &  56 &  63 &  69 &  97\\
50   & 101 & 103 & 103 & 144 & 143\\
100  & 191 & 190 & 192 & 189 & 185\\
200  & 243 & 246 & 240 & 233 & 221\\
500  & 465 & 474 & 465 & 446 & 416\\
1000 & 459 & 458 & 455 & 471 & 474 \\
\hline
\end{tabular}
}\hfil 
\subfigure[MATLAB.\label{tab:non_central_chi_2_times_matlab}]{
\begin{tabular}{|r|rrrrr|}
\multicolumn{1}{c}{\multirow{2}{*}{$ \lambda $}} & \multicolumn{5}{c}{$ \nu $} \\
\cline{2-6}
\multicolumn{1}{c|}{} & 1 &   5  &  10 &  50  & 100 \\
\hline
1& 168 & 214 & 259 & 456 & 294  \\ 
5& 651 & 782 & 840 & 1510 & 2046  \\ 
10& 935 & 1086 & 1050 & 1838 & 2496  \\ 
50& 3000 & 2969 & 2562 & 4118 & 5333  \\ 
100& 4929 & 3461 & 5039 & 6046 & 6299  \\ 
200& 9456 & 9603 & 10129 & 11524 & 12766  \\ 
500& 22691 & 22713 & 22702 & 23328 & 26273  \\ 
1000& 45872 & 43968 & 43807 & 44563 & 46780  \\
\hline
\end{tabular}
}\hfil \\[1em]

\caption{The computing ratio between the exact function and the piecewise linear approximation for the non-central $ \chi^2 $ distribution.}
\label{tab:implementation_times_chi}

\hfil 
\subfigure[C implementation (compared against CDFLIB).\label{tab:implementation_times_chi_c}]{
\begin{tabular}{|r|rrrrr|}
\multicolumn{1}{c}{\multirow{2}{*}{$ \lambda $}} & \multicolumn{5}{c}{$ \nu $} \\
\cline{2-6}
\multicolumn{1}{c|}{} & 1 &   5  &  10 &  50  & 100 \\
\hline
1    &    333&   412&   458&   666&   864\\
5    &    391&   447&   534&   701&   966\\
10   &    600&   668&   724&   801&   992\\
50   &   1411&  1424&  1231&  1811&  1811\\
100  &   2271&  2174&  2164&  2207&  2029\\
200  &   2539&  2624&  2791&  2304&  2113\\
500  &   5020&  4912&  4860&  4908&  4886\\
1000 &   4822&  4859&  4866&  4791&  4980\\
\hline
\end{tabular}
}\hfil
\subfigure[C++ implementation (compared against Boost).\label{tab:implementation_times_chi_cpp}]{
\begin{tabular}{|r|rrrrr|}
\multicolumn{1}{c}{\multirow{2}{*}{$ \lambda $}} & \multicolumn{5}{c}{$ \nu $} \\
\cline{2-6}
\multicolumn{1}{c|}{} & 1 &   5  &  10 &  50  & 100 \\
\hline
1    &    671&  1643&  1534&  1734&  2093\\
5    &   1884&  1831&  1733&  2037&  2344\\
10   &   1924&  1937&  1863&  2490&  2490\\
50   &   2576&  2565&  2876&  2945&  2974\\
100  &   3238&  3265&  3255&  3299&  3354\\
200  &   4382&  4384&  4373&  4356&  4333\\
500  &   5260&  5294&  5243&  5249&  5224\\
1000 &   6101&  6022&  6026&  6147&  6093\\
\hline
\end{tabular}
}\hfil\\[1em]

\caption{The RMSE of approximations $ \widetilde{C}^{-1}_{\nu}(\cdot;\lambda) $ to the non-central $ \chi^2 $ distribution.}
\label{tab:non_central_chi_2_rmse}

\hfil
\subfigure[Piecewise linear.\label{tab:non_central_chi_2_rmse_linear}]{
\begin{tabular}{|r|rrrrr|}
\multicolumn{1}{c}{\multirow{2}{*}{$ \lambda $}} & \multicolumn{5}{c}{$ \nu $} \\
\cline{2-6}
\multicolumn{1}{c|}{} & 1 &   5  &  10 &  50  & 100 \\
\hline
1    &  0.036 & 0.036 & 0.041 & 0.070 & 0.095 \\
5    &  0.045 & 0.047 & 0.050 & 0.076 & 0.100 \\
10   &  0.054 & 0.056 & 0.059 & 0.081 & 0.104 \\
50   &  0.098 & 0.099 & 0.101 & 0.116 & 0.133 \\
100  &  0.134 & 0.135 & 0.136 & 0.148 & 0.161 \\
200  &  0.186 & 0.187 & 0.188 & 0.196 & 0.207 \\
\hline
\end{tabular}
}\hfil
\subfigure[Piecewise cubic.\label{tab:non_central_chi_2_rmse_cubic}]{
\begin{tabular}{|r|rrrrr|}
\multicolumn{1}{c}{\multirow{2}{*}{$ \lambda $}} & \multicolumn{5}{c}{$ \nu $} \\
\cline{2-6}
\multicolumn{1}{c|}{} & 1 &   5  &  10 &  50  & 100 \\
\hline
1  &    0.004 & 0.005 & 0.006 & 0.007 & 0.006 \\
5  &     0.004 & 0.004 & 0.005 & 0.010 & 0.015 \\
10  &   0.006 & 0.005 & 0.005 & 0.009 & 0.014 \\
50  &    0.006 & 0.007 & 0.005 & 0.010 & 0.011 \\
100 &   0.013 & 0.008 & 0.009 & 0.011 & 0.014 \\
200 &   0.009 & 0.012 & 0.011 & 0.012 & 0.015 \\
\hline
\end{tabular}
\hfil
}\hfil\\[0em]

\hfil
\subfigure[\citeauthor{abdel1954approximate}.\label{tab:non_central_chi_2_rmse_abdel_aty}]{
\begin{tabular}{|r|rrrrr|}
\multicolumn{1}{c}{\multirow{2}{*}{$ \lambda $}} & \multicolumn{5}{c}{$ \nu $} \\
\cline{2-6}
\multicolumn{1}{c|}{} & 1 &   5  &  10 &  50  & 100 \\
\hline
1   &   0.153 & 0.041 & 0.023 & 0.009 & 0.006 \\
5   &   0.329 & 0.155 & 0.083 & 0.013 & 0.007 \\
10  &   0.385 & 0.243 & 0.156 & 0.026 & 0.011 \\
50   &  0.451 & 0.403 & 0.353 & 0.157 & 0.079 \\
100  &   0.461 & 0.435 & 0.405 & 0.250 & 0.157 \\
200  &  0.466 & 0.453 & 0.436 & 0.334 & 0.251 \\
\hline
\end{tabular}
}\hfil
\subfigure[\citeauthor{sankaran1959non}.\label{tab:non_central_chi_2_rmse_sankaran}]{
\begin{tabular}{|r|rrrrr|}
\multicolumn{1}{c}{\multirow{2}{*}{$ \lambda $}} & \multicolumn{5}{c}{$ \nu $} \\
\cline{2-6}
\multicolumn{1}{c|}{} & 1 &   5  &  10 &  50  & 100 \\
\hline
1  &       & 0.031 & 0.021 & 0.009 & 0.006 \\
5   &       & 0.058 & 0.031 & 0.010 & 0.007 \\
10  &   0.121 & 0.057 & 0.034 & 0.012 & 0.007 \\
50   &    0.030 & 0.025 & 0.020 & 0.012 & 0.009 \\
100  &  0.015 & 0.014 & 0.012 & 0.009 & 0.008 \\
200  &  0.008 & 0.007 & 0.007 & 0.006 & 0.006 \\
\hline
\end{tabular}
\hfil
}\hfil\\[0em]
\end{table}

\subsection{Approximating the non-central \texorpdfstring{$ \bm{\chi^2} $}{chi-squared} distribution}

There has been considerable research effort into quickly sampling from the non-central $ \chi^2 $ distribution by various approximating distributions \citep{johnson1995continuous,sankaran1959non,abdel1954approximate,wilson1931distribution,hoaglin1977direct,shea1991algorithm,best1975algorithm,rice1968uniform}, with the most notable being an approximation using Gaussian random variables by \citet{abdel1954approximate}, (and a similar type of approximation by \citet{sankaran1959non}). With the exception of \citet{abdel1954approximate} and \citet{sankaran1959non}, the remaining methods do not directly approximate the inverse cumulative distribution function, so do not give rise to an obvious coupling mechanism for multilevel Monte Carlo simulations, whereas our framework does. The schemes by \citet{abdel1954approximate} and \citet{sankaran1959non}, both of which require evaluating $ \Phi^{-1} $, are appropriate for comparison to our scheme. Hence, our approximation scheme, while modest in its sophistication, appears novel research in this direction.

For approximating the non-central $ \chi^2 $ distribution, we simplify our considerations by taking $ \nu $ to be fixed, and thus the distribution then only has the parameter $ \lambda $ varying. While this may appear a gross simplification, in the financial applications $ \nu $ is a model constant independent of the numerical method's parameters, and thus this simplification is appropriate.

We define the function $ P_x(\cdot;y) $ for $ x > 0 $ and $ 0 < y < 1  $ as 
\begin{equation*}
P_x(U;y) \coloneqq \sqrt{\dfrac{x}{4y}} \left( \dfrac{y}{x}  C^{-1}_{x}\left(U; \dfrac{(1 - y)x}{y}\right) - 1\right)
\qquad \text{so that} \qquad 
C^{-1}_{\nu}(U;\lambda) \equiv \lambda + \nu + 2 \sqrt{\lambda + \nu} P_\nu\left(U;\dfrac{\nu}{\lambda + \nu}\right).
\end{equation*}
This is because $ P_x(\cdot;y) $ is better scaled than $ C^{-1}_{\nu}(\cdot;\lambda) $ for the range of possible parameters, with the limits
\begin{equation*}
P_x(U;y) \xrightarrow{y\to 0} \Phi^{-1}(U) 
\qquad \text{and} \qquad 
P_x(U;y) \xrightarrow{y\to 1} \dfrac{C^{-1}_x(U)}{2\sqrt{x}} - \dfrac{\sqrt{x}}{2},
\end{equation*}
where $ C^{-1}_\nu(\cdot) $, without the parameter $ \lambda $, is the inverse cumulative distribution function of the central $ \chi^2 $ distribution. Thus to construct our approximation $ \widetilde{C}^{-1}_{\nu}(\cdot;\lambda) \approx C^{-1}_{\nu}(\cdot;\lambda) $, we first construct the piecewise polynomial approximation $ \widetilde{P}_x(\cdot;y) \approx P_x(\cdot;y) $, and then define 
$ \widetilde{C}^{-1}_{\nu}(\cdot;\lambda) \coloneqq \lambda + \nu + 2 \sqrt{\lambda + \nu} \widetilde{P}_\nu(\cdot;\tfrac{\nu}{\lambda + \nu}) $. An example piecewise linear approximation of $  \widetilde{C}^{-1}_{\nu}(\cdot;\lambda) $ using 8 intervals is shown in figure~\ref{fig:non_central_chi_squared_linear_approximation} for various values of the non-centrality parameter $ \lambda $ with the degrees of freedom fixed at $ \nu = 1 $. We can see from figure~\ref{fig:non_central_chi_squared_linear_approximation} that the fidelity of the approximation appears quite high across a range of parameter values. 

\begin{figure}[htb]
\centering

\hfil
\subfigure[Piecewise linear approximations to the non-central $ \chi^2 $ distribution with $ \nu = 1 $ and $ \lambda \in \{1,10,20\} $ using 8 intervals and 16 interpolation values.\label{fig:non_central_chi_squared_linear_approximation}]{\includegraphics{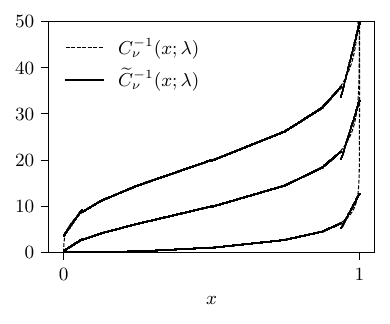}}\hfil
\subfigure[The variance of the underlying CIR process and various two way differences (using the same time increment): (\raisebox{-0.1em}{\huge$ \bullet$}) $ X_T $, ({\large $ \blacktriangledown $}) $ X_T - \widetilde{X}^{\mathrm{EM}}_T $, ($ \blacksquare $) $ X_T - \widetilde{X}^{\mathrm{linear}}_T $, and ($ \blacklozenge $) $ X_T - \widetilde{X}^{\mathrm{cubic}}_T $.\label{fig:cir_process_variance_reduction}]{\includegraphics{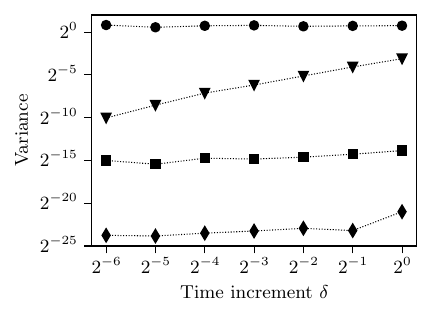}}\hfil
 
\caption{The non-central $ \chi^2 $ distribution from the CIR process, and the variance of path simulations and possible multilevel corrections.}
\label{fig:cir_process}

\end{figure}

There are only two difficulties with constructing such an approximation. The first is that the distribution is no longer anti-symmetric about $ u=1/2$, which is easily remedied by constructing two approximations: one for $ (0, \tfrac{1}{2}] $ and a second for $ (\tfrac{1}{2}, 1) $. The second difficulty is the parametrisation. Noting that $ y \in (0, 1) $ is dependent on the non-centrality parameter, we construct several approximations for various values of $ y $ (knowing the limiting cases) and linearly interpolate. A good choice of knot points are equally spaced values of $ \sqrt{y} $, where figure~\ref{fig:non_central_chi_squared_linear_approximation} uses 16 interpolation values.

The performance of C and C++ implementations are shown in table~\ref{tab:implementation_times_chi} on Intel Skylake hardware compiled with \texttt{g++}. Neither Intel, NAG, nor GSL offer an exact implementation of the inverse cumulative distribution function for the non-central $ \chi^2 $ distribution. Consequently, for our C implementation we compare against the \texttt{cdfchn} function from CDFLIB \citep{brown1994dcdflib,burkardt2020cdflib}, and for our C++ implementation the \texttt{quantile} function from the Boost library \citep{boost2020library} acting on a \texttt{non\_central\_chi\_squared} object (\texttt{boost/math/distributions}). Like the other implementations, Boost computes the inverse by a numerical inversion of the cumulative distribution function. We can see from table~\ref{tab:implementation_times_chi} that our approximation is orders of magnitude faster than the exact functions in both languages.

The fidelity of our non central $\chi^2$ approximations $ \widetilde{C}^{-1}_{\nu}(\cdot;\lambda) $ across the range of parameter values are quantified by the errors shown in table~\ref{tab:non_central_chi_2_rmse}. A piecewise linear approximation achieves a consistently good RMSE, as seen in table~\ref{tab:non_central_chi_2_rmse_linear}. Similarly, the piecewise cubic approximation demonstrates the same behaviour in table~\ref{tab:non_central_chi_2_rmse_cubic}, and has an RMSE which is typically an order of magnitude less than the piecewise linear approximation's. 

The approximations from \citet{abdel1954approximate} and \citet{sankaran1959non} are also shown in tables~\ref{tab:non_central_chi_2_rmse_abdel_aty} and \ref{tab:non_central_chi_2_rmse_sankaran} respectively. We can see from table~\ref{tab:non_central_chi_2_rmse_abdel_aty} that the approximation from \citeauthor{abdel1954approximate}, while improving for increasing $ \nu $, performs relatively poorly for large $ \lambda $. Conversely, the approximation by \citeauthor{sankaran1959non} in table~\ref{tab:non_central_chi_2_rmse_sankaran} is notably better than that by \citeauthor{abdel1954approximate} and often comparable to our piecewise cubic approximation, although demonstrating a much less consistent performance across the range of parameter values explored. 

Having compared the error coming from the approximation by \citet{sankaran1959non}, we can inspect the computational complexity of the approximation. The approximation can be expressed as
\begin{equation*}
\label{eqt:sankaran_approx}
\widetilde{C}^{-1}_{\nu}(U; \lambda) \coloneqq 
(k + \lambda) \left(1 + c_1 c_2\left(c_1 - 1 + \tfrac{1}{2}(2 - c_1) c_2 c_3\right) + c_1 \sqrt{2 c_3}\left(1 + \tfrac{1}{2}c_2 c_3\right) \Phi^{-1}(U) \right)^{{1}/{c_1}},
\end{equation*}
where $ c_1 \coloneqq 1 - \tfrac{2}{3}(k + \lambda)(k+3\lambda)(k + 2\lambda)^{-2} $, 
$ c_2 \coloneqq (k+2\lambda)(k + \lambda)^{-2} $, 
and $ c_3 \coloneqq (c_1 - 1)(1 - 3c_1) $. Just as in our polynomial approximations, it contains an identical mix of basic arithmetic operations, but a $ \Phi^{-1} $ evaluation. This can either be calculated exactly, or for improved speed it can potential be further approximated itself, but this then compounds a second level of approximation on top of the original. Furthermore, the expression contains an exponentiation by a factor of $ 1/c_1 $, which will not in general be an integer. It can be expected that such a non-integer exponentiation operation will be quite expensive. Consequently, we anticipate this would be measurably slower than our polynomial approximations. 

Contrasting the approximations between \citet{sankaran1959non} and our piecewise cubic, the piecewise cubic demonstrates a more consistent RMSE alongside a simpler computational composition, makes our piecewise polynomial the more attractive of the two. 

\subsection{Simulating the Cox-Ingersoll-Ross process}

We mentioned the non-central $ \chi^2 $ distribution arises from the Cox-Ingersoll-Ross (CIR) process: $ \dd{X_t} = \kappa (\theta - X_t) \dd{t} + \sigma \sqrt{X_t} \dd{W_t} $ for strictly positive parameters $ \kappa $, $ \theta $ and $ \sigma $. The distribution of $ X_T $ is a scaled non-central $ \chi^2 $ distribution with $ \nu = \tfrac{4\kappa\theta}{\sigma^2} $ and $ \lambda = \tfrac{4\kappa X_0 \exp(-\kappa T)}{\sigma^2(1 - \exp(-\kappa T))} $ \citep{cox1985theory} \citep[pages~67--68]{munk2011fixed}. Simulating this with the Euler-Maruyama scheme forcibly approximates the non-central $ \chi^2 $ distribution as a Gaussian distribution, and can require very finely resolved path simulations to bring the bias down to an acceptable level, as explored by \citet{broadie2006exact}. Similarly the Euler-Maruyama scheme is numerically ill posed due to the $ \sqrt{X_t} $ term, and several adaptions exist to handle this appropriately \citep{deelstra1998convergence,lord2010comparison,berkaoui2008euler,higham2002strong,alfonsi2005discretization,alfonsi2008second,alfonsi2010high,dereich2012euler,cozma2020strong_euler,gyongy2011note}, especially when the Feller condition $ 2\kappa\theta \geq \sigma^2 $ is not satisfied \citep{feller1951two,gyongy1998note}. The Feller condition is significant because \citep[page~391]{cox1985theory}: ``\textit{$ [X_t] $} can reach zero if $ \sigma^2 > 2\kappa \theta $. If $ 2 \kappa \theta \geq \sigma^2 $, the upward drift is sufficiently large to make the origin inaccessible [see \citet{feller1951two}]''.

Rather than approximating the non-central $ \chi^2 $ distribution with an exact Gaussian distribution when using the Euler-Maruyama scheme, we propose using approximate non-central $ \chi^2 $ random variables, such as those from a piecewise linear approximation. This has the benefit of offering vast time savings, whilst introducing far less bias than the Euler-Maruyama scheme. The piecewise linear approximation is on average hundreds of times faster than the exact function, as was seen in table~\ref{tab:non_central_chi_2_times}, giving vast savings, while still achieving a high fidelity, as was seen from table~\ref{tab:non_central_chi_2_rmse}.

We can generate path simulations of the CIR process using the exact non-central $ \chi^2 $ distribution, piecewise linear and cubic approximations, and the truncated Euler-Maruyama scheme from \citet{higham2002strong} (where $ \sqrt{X_n} \to \sqrt{\lvert X_n \rvert} $), where we set $ \kappa= \tfrac{1}{2}$ and  $ \theta = \sigma =T = X_0 = 1 $ (just satisfying the Feller condition\footnote{In this case $ \nu = 2 $ for which there is a finite and non zero probability density at zero.}). These respectively give rise to the exact underlying process $ X_T $ and the three approximations $ \widetilde{X}^{\mathrm{linear}}_T $, $ \widetilde{X}^{\mathrm{cubic}}_T $, and 
$ \widetilde{X}^{\mathrm{EM}}_T $.
The variances of the underlying process and the differences between this and the approximations are shown in figure~\ref{fig:cir_process_variance_reduction}. Note that the difference terms are evaluated using the same time increments, and there is no mixing of coarse and fine paths in these two-way differences. The results of figure~\ref{fig:cir_process_variance_reduction} do not change if we instead satisfy the Feller condition with strict inequality by setting e.g.\ $ \kappa = \tfrac{3}{2} $.

From figure~\ref{fig:cir_process_variance_reduction} we can see the exact underlying process' variance does not vary with the discretisation, which is to be expected as the simulation is from the exact distribution, and thus has no bias. We can see that the Euler-Maruyama approximation to the process exhibits a strong convergence order $ \tfrac{1}{2} $, as expected \citep{higham2002strong,gyongy1998note}. For the piecewise linear approximation, as the fidelity will not vary with the discretisation, then subsequently the variance also does not vary. Furthermore, given the high fidelity seen in table~\ref{tab:non_central_chi_2_rmse_linear}, the drop in the variance is approximately $ 2^{-15} $. Similarly, the even higher fidelity cubic approximation demonstrates the same behaviour in table~\ref{tab:non_central_chi_2_rmse_cubic}, but with a drop in variance of $ 2^{-25} $. These findings demonstrate that the most substantial variance reductions come from using our piecewise polynomial approximations to the non central $ \chi^2 $ distribution's inverse cumulative distribution function, rather than using the Gaussian distribution with the Euler-Maruyama scheme. 

We can repeat these same variance comparisons with a different functional of the terminal value which is path dependent. A common example in financial simulations to consider the arithmetic average of the path, (an \textit{Asian option} in mathematical finance \citep{glasserman2013monte}). Doing so produces an exactly identical variance structure as seen in figure~\ref{fig:cir_process_variance_reduction}. From this we conclude that our findings seen and explained thus far appear to carry over immediately to simulations requiring an entire path simulation, not only a terminal value, justifying the use of our approximations in time marching schemes.

Quantifying the savings that can be expected, we conservatively approximate the cost savings for both the linear and cubic approximations as a factor of 300 using table~\ref{tab:implementation_times_chi}. Coupling this with the variance reductions seen in figure~\ref{fig:cir_process_variance_reduction}, the anticipated cost savings are shown in table~\ref{tab:savings_chi}. We can see from this that both offer impressive time savings, and the piecewise cubic again achieves a near perfect efficiency. However, the piecewise linear approximation's efficiency, while good, is losing an appreciable fraction, making the piecewise cubic the preferred choice. Ultimately, both offer vast time savings by factors of 250 or higher.

\begin{table}[htb]
\centering
\caption{The cost savings, variance reductions, and possible speed ups from approximating the non-central $ \chi^2 $ distribution.}
\label{tab:savings_chi}
\begin{tabular}{lcclrcr}
Approximation  & $ {\log}_2 \left(\tfrac{\widetilde{V}_l}{\hat{v}_l}\right) $ & $ \tfrac{\hat{c}_l}{\tilde{c}_l} $ & \multicolumn{2}{c}{Speed up} & $ \tfrac{\widetilde{m}_l}{\widehat{m}_l} $ & \multicolumn{1}{c}{$ \tfrac{\widetilde{m}_l}{\widetilde{M}_l} $} \\[0.5em]
\hline
Piecewise linear  & $ -15 $ & 300 & 249 & (83.3\%) & 1.09 & 2900 \\
Piecewise cubic  & $ -25 $ & 300 & 298 & (99.4\%)& 1.00 & 100000  
\end{tabular}
\end{table}

\section{Conclusions}
\label{sec:conclusions}

The expense of sampling random variables can be significant. In the work presented, we proposed, developed, analysed, and implemented approximate random variables as a means of circumventing this for a variety of modern hardwares. By incorporating these into a nested multilevel Monte Carlo framework we showcased how the full speed improvements offered can be recovered with near perfect efficiency without losing any accuracy. With a detailed treatment, we showed that even for basic simulations of geometric Brownian motions requiring Gaussian random variables, speedups of a factor 5 or more can be expected. The same framework was also applied to the more difficult Cox-Ingersoll-Ross process and its non-central $ \chi^2 $ distribution, offering the potential for vast speedups of a factor of 250 or more. 

For sampling from a wide class of univariate distributions, we based our work on the inverse transform method. Unfortunately, this method is expensive as it relies on evaluating a distribution's inverse cumulative distribution function. We showed that many implementations for the Gaussian distribution, while accurate to near machine precision, are ill-suited in several respects for modern vectorised hardware. To address this, we introduced a generalised notion of approximate random variables, produced using approximations to a distribution's inverse cumulative distribution function. The two major classes of approximations we introduced were: piecewise constant approximations using equally spaced intervals, and piecewise linear approximations using geometrically small intervals dense near the distribution's tails. These cover a wide class of possible approximations, and notably recover as special cases: Rademacher random variables and the weak Euler-Maruyama scheme \citep{glasserman2013monte}, moment matching schemes by \citet{muller2015improving}, and truncated bit approximations by \citet{giles2019random_quadrature}. For piecewise constant and linear approximations, we analysed and bounded the errors for a range of possible norms. The significance of these bounds is that they are valid for arbitrarily high moments, which extends the results from \citet{giles2019random_quadrature}, and crucially is necessary for the nested multilevel Monte Carlo analysis by \citet{giles2020approximate}. Lastly, these approximations were able to achieve very high fidelities, with the piecewise linear approximation using a geometric sequence of intervals providing high resolution of singularities.

With the approximations detailed from a mathematical perspective, we highlighted two possible implementations in C \citep{sheridan2020approximate_random}. The benefit of these approximations is that they are by design ideally suited for modern vector hardware and achieve the highest possible computational speeds. They can be readily vectorised using OpenMP SIMD directives, have no conditional branching, avoid division and expensive function evaluations, and only require simple additions, multiplications, and bit manipulations. The piecewise constant and linear implementations were orders of magnitude faster than most freely available open source libraries, and typically a factor of 5--7 times faster than the proprietary Intel Maths Kernel Library, achieving close to the maximum speed of reading and writing to memory. This speed comes from the simplicity of their operations, and heavy capitalisation on the fast speeds of querying the cache and vector registers. 

Incorporating approximate random variables into the nested multilevel Monte Carlo framework by \citet{giles2020approximate}, the low errors and fast speeds of the approximations can be exploited to obtain their full speed benefits without losing accuracy. Inspecting the appropriate multilevel variance reductions, we demonstrated how practitioners can expect to obtain speed improvements of a factor of 5--7 by using approximate Gaussian random variables, where the fastest approximation was the piecewise linear approximation. This appears to be the case when using either the Euler-Maruyama or Milstein schemes. 

Considering the Cox-Ingersoll-Ross process \citep{cox1985theory}, this is known to give rise to the non-central $ \chi^2 $ distribution, which is an example of a parametrised distribution, and is also extremely expensive to sample from, as we demonstrated. We applied our approximation framework to this to produce a parametrised approximation. The error of using our approximate random variables was orders of magnitude lower than approximating paths using the Euler-Maruyama scheme, showing our approximate random variables are considerably more suitable for generating path simulations than the Euler-Maruyama scheme. This circumvents the problem of the Euler-Maruyama scheme having a very large bias for the Cox-Ingersoll-Ross process \citep{broadie2006exact}. Furthermore, our implementation was substantially quicker than those by CDFLIB and Boost, offering a speed improvement of 250 times or higher. 

All of the code to produce these figures, tables, and approximations is freely available and hosted in repositories by \citet{sheridan2020approximate_inverse,sheridan2020approximate_random}, along with an introductory guide aimed at practitioners \citep{sheridan2020approximate_random}.

\section{Acknowledgements}

We would like to acknowledge and thank those who have financially sponsored this work. This includes the Engineering and Physical Sciences Research Council (EPSRC) and Oxford University's centre for doctoral training in Industrially Focused Mathematical Modelling (InFoMM), with the EP/L015803/1 funding grant. Furthermore, this research stems from a PhD project \citep{sheridan2020nested} which was funded by Arm and NAG. Funding was also provided by the EPSRC ICONIC programme grant EP/P020720/1, the Hong Kong Innovation and Technology Commission (InnoHK Project CIMDA), and by Mansfield College, Oxford.  

\bibliography{references}

\end{document}